%% file: ex_article.tex
\documentclass[onefignum,onetabnum]{siamart190516}

\input{ex_shared}
\ifpdf
\hypersetup{
	pdftitle={Additive Schwarz methods for serendipity elements},
	pdfauthor={J. Marchena-Menendez, and R. C. Kirby}
}
\fi

\begin{document}
	
	\maketitle
	
	\begin{abstract}
		While solving Partial Differential Equations (PDEs) with Finite Element Methods (FEM), serendipity elements allow us to obtain the same order of accuracy as rectangular tensor-product elements with many fewer degrees of freedom (DoFs).  To realize the possible computational savings, we develop some additive Schwarz methods (ASM) based on solving local patch problems.  Adapting arguments from Pavarino for the tensor-product case, we prove that patch smoothers give conditioning estimates independent of the polynomial degree for a model problem.  We also combine this with a low-order global operator to give an optimal two-grid method, with conditioning estimates independent of the mesh size and polynomial degree.  The theory holds for serendipity elements in two and three dimensions, and can be extended to multigrid algorithms.  Numerical experiments using Firedrake and PETSc confirm this theory and demonstrate efficiency relative to standard elements.
	\end{abstract}
	
	
	
	\section{Introduction}\label{sec:intro}
	
	A very popular method for solving Partial Differential Equations (PDEs) is the known Finite Element Method (FEM).
	These methods approximate solutions in finite-dimensional spaces consisting of piecewise polynomials or other suitable functions defined over some subdivision of the domain into simple shapes.
	Of current research interest is the serendipity space $\mathcal{S}_k$, which contains all of standard polynomials of total degree $k$ (hence denoted by $\mathcal{P}_k$), with enough additional polynomials to enforce $C^0$ continuity on quadrilaterals.
	This gives the same order of accuracy as tensor product  spaces $\mathcal{Q}_k$ with a much smaller approximating space, which may save time and resources when solving the PDEs with FEM.
	
	An important and general reference on the study on serendipity elements is the paper by Arnold and Awanou \cite{Arnold}. 
	Recently, Gillette et al.~\cite{Gillette} expanded the theory beyond $H^1$ elements to differential forms. 
	The full convergence theory for these elements requires rectangular or at least affinely-mapped quadrilateral/hexahedral elements.  Curvilinear mapping can degrade the order of convergence for serendipity elements~\cite{arnold2015finite} in particular cases, although triggering this pathology seems quite subtle, and we include some computational results on a general quadrangulation.
	Serendipity spaces that guarantee full accuracy on general quadrilaterals are known~\cite{arbogast2022direct}, but these spaces include some rational functions.  Our analysis in does not cover this case, although it is a possible extension.
	
	While using higher-degree polynomials can lead to very accurate solutions, it is also important to consider the impact that they have on solving the resulting algebraic systems.
	In particular, both the size of the system and its condition number can increase dramatically as the degree increases.
	In a multigrid framework, classical smoothers like Jacobi iteration degrade or even fail as the degree increases.
	In order to obtain solvers that are uniform in the polynomial degree, we turn to the class of \emph{additive Schwarz} methods.  
	Pavarino~\cite{Pavarino1993/94} develops such techniques for high-order discretizations on classical $\mathcal{Q}_k$ elements.
	Here, one uses small, overlapping subdomains consisting of cells sharing each vertex in the mesh.
	An additive Schwarz preconditioner then approximates the inverse of the finite element operator by the sum of the inverses of these restrictions to patches, hence called a \emph{patch smoother}.
	Pavarino (by a rather lengthy technical argument) gives estimates for these methods that are in fact uniform in the polynomial degree.
	In order to obtain a preconditioner that is also uniform in the mesh parameter $h$, one can either include the lowest order $\mathcal{Q}_1$ space over the original mesh in the subspace decomposition or else use the patch smoother on each level of a multigrid hierarchy.
	Similar estimates for simplicial $\mathcal{P}_k$ elements are given in~\cite{Schoeberl:2008}, but neither analysis directly covers the case of serendipity elements.
	However, the reduced cardinality of the $\mathcal{S}_k$ space relative to $\mathcal{Q}_k$ suggests that such additive Schwarz methods could be a very powerful tool to combine with  serendipity elements.
	We refer the reader to references such as~\cite{arnold2000multigrid,Schoeberl1999}, noting that for some problems, multigrid requires patch-based Schwarz smoothers even for low-order discretizations.
	
	In this paper, we provide the missing technical justification of these additive Schwarz methods in the context of serendipity spaces and explore their practical properties.  \Cref{sec:main} contains a description of the model problem and its variational formulation and approximation spaces.   
	Then, \cref{sec:devitermeth} presents the resulting preconditioners and their analysis.  \cref{subsec:asloworder} provides the key theoretical developments of the paper.  
	By a judicious use of the Ritz projection~\cite{Thomee,wheeler1973apriori}, we are able to extend Pavarino's technical results from the $\mathcal{Q}_k$ spaces to serendipity elements.
	This allows us to prove that the patch smoother behaves uniformly in the degree $k$.
	Results for the patch smoother itself allow us to give optimal estimates both for a two-level additive Schwarz and a multigrid method.
	\Cref{sec:num} provides a suite of numerical experiments obtained using the Firedrake project~\cite{Rathgeber:2016}, and then we offer concluding remarks in~\cref{sec:conc}.

	\section{Serendipity Elements}\label{sec:main}
	We consider a model problem for linear, self-adjoint, second order elliptic problems on a bounded Lipschitz region $\Omega \subset\mathbb{R}^n$ with $n=2, 3$.
	
	We consider the following example with Dirichlet boundary conditions imposed on $\Gamma_{\mathrm{D}},$ a closed subset of $\partial \Omega$ with positive measure, and homogeneous Neumann conditions on the rest of the boundary $\Gamma_{\mathrm{N}}$.  Our techniques can be readily adapted to other boundary conditions.  The boundary value problem is
	$$
	\left\{\begin{aligned}
		-\Delta u &=& f & & \text { in } \Omega, \\
		u &=& u_{0} & & \text { on } \Gamma_{\mathrm{D}}, \\
		\frac{\partial u}{\partial n} &=& g & & \text { on } \Gamma_{\mathrm{N}}.
	\end{aligned}\right.
	$$
	
	We suppose that $f$ and $u_{0}$ satisfy standard regularity assumptions and that $u_{0}=0$.   Define the subspace $H_D^1(\Omega) = \left\{v \in H^{1}(\Omega): v|_{\Gamma_{\mathrm{D}}} = 0\right\} \equiv V$ and affinely shifted set
        $H_{D,u_0}^1(\Omega) = \left\{v \in H^{1}(\Omega): v|_{\Gamma_{\mathrm{D}}} = u_0\right\} \equiv V_{D}$.  The variational form of our problem is to find $u^* \in V_{D}$ such that 
	\begin{align} \label{ref1}
		a\left(u^{*}, v\right)=F(v), \quad \forall v \in V,
	\end{align}
	where
	\[
	a(u, v)=\int_{\Omega} \nabla u \cdot \nabla v dx \quad \text { and } \quad F(v)=\int_{\Omega} f v d x +\int_{\Gamma_{\mathrm{N}}} g v ds.
	\]

	Although our theory is worked out for scalar coercive problems such as~\eqref{ref1}, it generalizes readily to symmetric coercive vector-valued problems such as planar elasticity~\cite{brennerandscott}.
	This problem determines the elastic deformation $u$ of a reference domain $\Omega$ subject to some loading function $f$.	
	Let $\lambda$ and $\mu$ be the Lam\'e constants.  For some vector displacement field $u$, $\varepsilon(u) = \frac{1}{2} \left( \nabla u + \nabla u^T\right)$ is the symmetric gradient.  The plane stress associated with $u$ is
	\[
	\sigma(u) = \lambda tr(\varepsilon(u)) I + 2 \mu \varepsilon(u).
	\]
	
	Here, we take $V$ to be the subspace of $(H^1(\Omega))^2$ with functions vanishing on the Dirichlet boundary.  With these definitions, the variational problem for elastic displacement is to find $u$ in $V$ such that
	\begin{equation}
		\label{elasticity}
		a(u, v) = (f, v), \ \ \ v \in V,
	\end{equation}
	where
	\begin{equation}
		a(u, v) = 
		\int_\Omega \sigma(u) : \varepsilon(v)  dx = \int_\Omega 2 \mu \epsilon(u) : \epsilon(v) + \lambda \operatorname{div}(u) \operatorname{div}(v) dx.
	\end{equation}
	
	As in~\cite{Schoeberl:2008}, the theory and practice of the additive Schwarz method carries over directly into this problem.
	
	We let $\mathcal{T}_{h>0}$ denote a decomposition of $\Omega$ into rectangles (or, more generally, quadrilaterals) of size $h$ such that intersection of the closures of any two distinct elements is empty or else contains either a single vertex or an entire edge.  Over each single mesh entity $K$, we define $\mathcal{P}_k(K)$ to be the space of polynomials over $K$ of total degree no greater than $k$.  This space has dimension of
	\begin{equation}
		\dim \mathcal{P}_k(K) = \binom{k+n}{n}
		= \begin{cases} \frac{(k+1)(k+2)}{2} &  n=2, \\
			\frac{(k+1)(k+2)(k+3)}{6}& n=3.
		\end{cases}
	\end{equation}
	We also let
	$\mathcal{Q}_k(K)$ be the set of polynomials of degree $k$ in each variable separately.  The dimensions of these spaces satisfy:
	\begin{equation}
		\dim \mathcal{Q}_k(K) = \left(k+1\right)^n.
	\end{equation}
	
	Our main interest here is the \emph{serendipity} space~\cite{Arnold}, which includes all of $\mathcal{P}_k(K)$ and hence offer optimal approximation properties, but has much lower dimensionality than $\mathcal{Q}_k(K)$.
	Hence, they offer comparable accuracy to standard methods, but at a lower cost.
	When $n=2$,  the local serendipity space is defined by:
	\begin{equation}
		\label{eq:sk}
		\mathcal{S}_k(K) = \mathcal{P}_k(K) + \text{span}\{ x^k y, x y^k \},
	\end{equation}
	for each $k \geq 2$.  This has dimension of $\dim \mathcal{P}_k(K) + 2 = \frac{1}{2}(k^2-k+1)$.  For $k=1$, the serendipity space $\mathcal{S}_k(K)$ is taken to coincide with $\mathcal{Q}_1(K)$.  Members of the space may be parameterized by their vertex values, moments of degree $k-1$ along each edge, and for $k\geq 4$, moments against $\mathcal{P}_{k-4}(K)$ as depicted in \cref{fig:dofs2d} for degrees 2 and 3.
	
	\begin{figure}[htbp]
		\qquad\qquad
		\begin{subfigure}[t]{0.4\textwidth}
			\begin{center}
				\begin{tikzpicture}[scale=1.1]
					
					\foreach \x/\y/\i in
					{-1/1/a,1/1/b,1/-1/c,-1/-1/d} {
						\coordinate (\i) at (\x, \y);
					}
					
					\foreach \i/\j in {a/b,b/c,c/d,d/a} {
						\draw (\i)--(\j);
					}
					
					\foreach \pt in {a,b,c,d}{
						\draw[fill] (\pt) circle (0.05);
					}
					
					\foreach \i/\j in {a/b,b/c,c/d,d/a} {
						\draw[fill] ($0.5*(\i)+0.5*(\j)$) circle (0.05);
					}
					
				\end{tikzpicture}
			\end{center}
		\end{subfigure}
		\begin{subfigure}[t]{0.4\textwidth}
			\begin{center}
				\begin{tikzpicture}[scale=1.1]
					
					\foreach \x/\y/\i in
					{-1/1/a,1/1/b,1/-1/c,-1/-1/d} {
						\coordinate (\i) at (\x, \y);
					}
					
					\foreach \i/\j in {a/b,b/c,c/d,d/a} {
						\draw (\i)--(\j);
					}
					
					\foreach \pt in {a,b,c,d}{
						\draw[fill] (\pt) circle (0.05);
					}
					
					\foreach \i/\j in {a/b,b/c,c/d,d/a} {
						\foreach \alph/\bet in {0.333/0.667,0.667/0.333} {
							\draw[fill] ($\alph*(\i)+\bet*(\j)$) circle (0.05);
						}
					} 
				\end{tikzpicture}
			\end{center}
		\end{subfigure}\\
		
		\qquad\qquad
		\begin{subfigure}[t]{0.4\textwidth}
			\begin{center}
				\begin{tikzpicture}[scale=1]
					\foreach \x/\y/\z/\i in
					{-1/-1/1/a,1/-1/1/b,1/1/1/c,-1/1/1/d, -1/-1/-1/e, 1/-1/-1/f, 1/1/-1/g, -1/1/-1/h} {
						\coordinate (\i) at (\x, \y, \z);
					}
					\foreach \i/\j in {a/b,b/c,c/d,d/a,b/f,f/g,g/c,d/h,g/h} {
						\draw (\i)--(\j);
					}
					
					\foreach \pt in {a,b,c,d,f,g,h}{
						\draw[fill] (\pt) circle (0.05);
					}      
					
					\foreach \i/\j in {a/b,b/c,c/d,d/a,f/g,g/h,b/f,c/g,d/h} {
						\draw[fill] ($0.5*(\i)+0.5*(\j)$) circle (0.05);
					}
				\end{tikzpicture}
			\end{center}
			
		\end{subfigure}
		\begin{subfigure}[t]{0.4\textwidth}
			\begin{center}
				\begin{tikzpicture}[scale=1]
					\foreach \x/\y/\z/\i in
					{-1/-1/1/a,1/-1/1/b,1/1/1/c,-1/1/1/d, -1/-1/-1/e, 1/-1/-1/f, 1/1/-1/g, -1/1/-1/h} {
						\coordinate (\i) at (\x, \y, \z);
					}
					\foreach \i/\j in {a/b,b/c,c/d,d/a,b/f,f/g,g/c,d/h,g/h} {
						\draw (\i)--(\j);
					}
					
					\foreach \pt in {a,b,c,d,f,g,h}{
						\draw[fill] (\pt) circle (0.05);
					}      
					
					\foreach \i/\j in {a/b,b/c,c/d,d/a,f/g,g/h,b/f,c/g,d/h} {
						\draw[fill] ($0.333*(\i)+0.667*(\j)$) circle (0.05);
						\draw[fill] ($0.667*(\i)+0.333*(\j)$) circle (0.05);
					}
				\end{tikzpicture}
			\end{center}
		\end{subfigure}

		\caption{Canonical degrees of freedom for $\mathcal{S}_2(K)$ (left) and $\mathcal{S}_3(K)$ (right).}
		\label{fig:dofs2d}

	\end{figure}
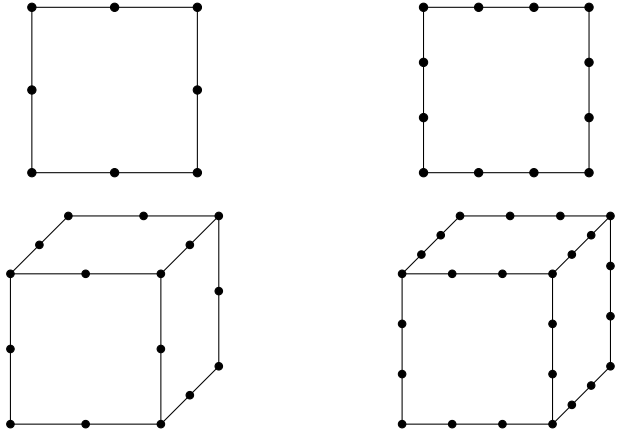

	Following~\cite{Arnold}, the generalization of~\eqref{eq:sk} to higher spatial dimensions uses the notion of \emph{superlinear degree}.  A monomial in $n$ variables may be written as $p = \Pi_{i=1}^n x_i^{\alpha_i}$ for integers $\alpha_i \geq 0$, and then the superlinear degree is given by the sum of each $\alpha_i$ that is at least 2:
	\begin{equation}
		\text{deg}_2(p) = \sum_{\stackrel{i=1}{\alpha_i > 1}}^n \alpha_i.
	\end{equation}
	
	Then, the serendipity space of degree $k$ is defined as all monomials with superlinear degree at most $k$:
	\begin{equation}
		\mathcal{S}_k(K) = \{ p \in \mathcal{Q}_k(K) : \text{deg}_2(p) \leq k \}.
	\end{equation}
	
	This definition coincides with that given in \eqref{eq:sk}. In three variables, the space $\mathcal{S}_2$ is $\mathcal{P}_2$ together with the span of:
	\[
	\{ xyz, x^2y, x^2z, x^2 yz, x y^2, y^2 z, x y^2 z, x z^2, y z^2, xy z^2 \},
	\]
	and $\dim \mathcal{S}_2(K) = 20$.  While larger than $\dim{P}_2(K) = 10$, it still is smaller than $\dim{Q}_2(K) = 27$.
	
	The dimension of the general serendipity spaces does not admit a simple formula in more than two dimensions, but it is shown in~\cite{Arnold} to be
	\begin{equation}
		\dim \mathcal{S}_k(K) = \sum_{d=0}^{\min(n,\myfloor{k/2})} 2^{n-d}\binom{n}{d}\binom{k-d}{d}.
	\end{equation}

	\begin{table}
		\begin{center}
			\begin{tabular}{c|c}
				$n=2$ & $n=3$ \\ \hline
				\begin{tabular}{r|ccccc}
					$k$: & 1 & 2 & 3 & 4 & 5 \\ \hline
					$\dim \mathcal{P}_{k}$ & 3 & 6 & 10 & 15 & 21 \\
					$\dim \mathcal{S}_{k}$ & 4 & 8 & 12 & 17 & 23 \\
					$\dim \mathcal{Q}_{k}$ & 4 & 9 & 16 & 25 & 36
				\end{tabular} &
				\begin{tabular}{r|ccccc}
					$k$: & 1 & 2 & 3 & 4 & 5 \\ \hline
					$\dim \mathcal{P}_{k}$ & 4 & 10 & 20 & 35 & 56  \\
					$\dim \mathcal{S}_{k}$ & 8 & 20 & 32 & 50 & 74  \\
					$\dim \mathcal{Q}_{k}$ & 8 & 27 & 64 & 125 & 216        
				\end{tabular}
			\end{tabular}
		\end{center}
		\label{table:dims}
		\caption{Dimension of local polynomial spaces in two and three variables.}
	\end{table}

	The global finite element space based on serendipity elements is 
	\begin{equation}
		V_{h,k} = \left\{ v \in C^0(\Omega) : v|_{K} \in \mathcal{S}_k(K) \text{ for each } K \in \mathcal{T}_h \right\}.
	\end{equation}
	
	By means of the inner product $(.,.)$ defined as the one of $L^2(V)$, the bilinear form $a$ in~\eqref{ref1}  induces an operator $A_{h,k} : V_{h,k} \rightarrow V_{h, k}^\prime$ by
	\begin{equation}
		\left(A_{h,k} u , v\right) = a(u, v),
	\end{equation}
	where the duality pairing is that between $V_{h,k}$ and its dual inherited from the inclusion $V_{h,k} \hookrightarrow H_0^1(\Omega)$.

	\section{Developing iterative methods}\label{sec:devitermeth}
	In this paper, we consider optimal preconditioners for serendipity finite element discretizations of our model problem. 
	These are based on additive Schwarz or subspace decomposition techniques that approximate the inverse of an operator by restricting that operator to a collection of subspaces and summing the inverses of those restrictions.
	In our case, we rely primarily on subspaces consisting of patches of cells sharing a common vertex in a finite element mesh.
	Such techniques are useful to obtain degree-independence for $\mathcal{Q}_k$ and $\mathcal{P}_k$ spaces~\cite{Pavarino1993/94,Schoeberl:2008}, and are also required for defining multigrid smoothers in $\mathcal{H}(\text{div})$ and $\mathcal{H}(\text{curl})$, even for low-order spaces~\cite{arnold2000multigrid, ArnoldFalkWinther}.
	We define these smoothers, which give degree-independent conditioning estimates for serendipity spaces, in~\cref{patchsmoothers}.
	In order to obtain a preconditioner that is also independent of the mesh spacing, we can include a low-order subspace on the original mesh, as described in~\cref{subsec:asloworder}.
	Alternatively, we can use the patch smoother on each level of a geometric multigrid hierarchy, as we describe in~\cref{ssec:mg}.
	
	The cost of building and solving these local patch problems is the dominant cost in the algorithms we consider.
	Since they solve finite element operators over a (very small) mesh, the local systems contain sparsity, but they are small enough that storing and working with dense matrices may be preferable in practice.
	In either case, using $\mathcal{S}_k$ rather than $\mathcal{Q}_k$ elements can greatly reduce the size of these local problems and hence prove favorable for run-time.
	
	\subsection{Patch smoothers}\label{patchsmoothers}
	We let $\{\mathbf{v}_i\}_{i=1}^N$ denote the set of interior vertices of the mesh, i.e. those vertices that do not belong to any boundary edges of the mesh, and let
	$\Omega^i$ denote the set of mesh cells of which $\mathbf{v}_i$ is a member of the closure.  Typically, this will consist of a patch of 4 rectangles in 2D or 8 boxes in 3D (see \cref{fig:patches}).  
	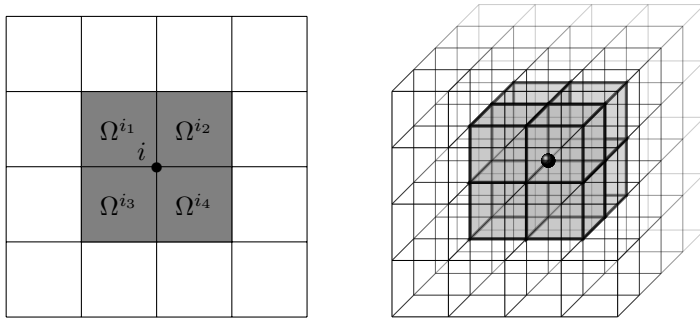
\begin{figure}[htbp]
		\qquad\qquad
		\begin{subfigure}[t]{0.4\textwidth}
			\begin{center}
				\begin{tikzpicture}[scale=2.0]
					\draw[step=0.5cm] (-1,-1) grid (1,1);
					\node[rectangle,draw=black, minimum size=1cm, fill=gray] at (-0.25,+0.25){}; 
					\node at (-0.25,+0.25) {$\Omega^{i_1}$};
					\node[rectangle,draw=black, minimum size=1cm, fill=gray] at (+0.25,+0.25){}; 
					\node at (+0.25,+0.25) {$\Omega^{i_2}$};
					\node[rectangle,draw=black, minimum size=1cm, fill=gray] at (-0.25,-0.25){}; 
					\node at (-0.25,-0.25) {$\Omega^{i_3}$};
					\node[rectangle,draw=black, minimum size=1cm, fill=gray] at (+0.25,-0.25){}; 
					\node at (+0.25,-0.25) {$\Omega^{i_4}$};
					\node at (+0,-0.01) {$\bullet$};
					\node at (-0.1,+0.1) {$i$};
				\end{tikzpicture}
			\end{center}
			
		\end{subfigure}
		\begin{subfigure}[t]{0.4\textwidth}
			\begin{center}
				\begin{tikzpicture}[scale=0.75,
					thin,cube/.style={very thick,fill ,fill=gray!50}]
					
					\foreach \x in {1,2}
					\foreach \y in {1,2}
					\foreach \z in {1,2}{
						\filldraw[cube,opacity=.20 *  (\z + 1)  ] (\x,\y,\z) -- (\x,\y+1,\z) -- (\x+1,\y+1,\z) -- (\x+1,\y,\z) -- cycle;
						
						\filldraw[cube,opacity=.20 *  (\z + 1)  ] (\x,\y,\z) -- (\x,\y+1,\z) -- (\x,\y+1,\z+1) -- (\x,\y,\z+1) -- cycle;
						
						\filldraw[cube,opacity=.20 *  (\z + 1) ]  (\x,\y,\z) --  (\x+1,\y,\z) --  (\x+1,\y,\z+1) --  (\x,\y,\z+1) -- cycle;
						
						\filldraw[cube,opacity=.20 *  (\z + 1)  ] (\x+1,\y,\z) -- (\x+1,\y+1,\z) -- (\x+1,\y+1,\z+1) -- (\x+1,\y,\z+1) -- cycle;
						
						\filldraw[cube,opacity=.20 *  (\z + 1)  ] (\x,\y+1,\z)  -- (\x+1,\y+1,\z) -- (\x+1,\y+1,\z+1) -- (\x,\y+1,\z+1) -- cycle;
						
						\filldraw[cube,opacity=.20 *  (\z + 1)  ] (\x,\y,\z+1) -- (\x,\y+1,\z+1)  -- (\x+1,\y+1,\z+1) -- (\x+1,\y,\z+1) -- cycle;

					}
					
					\shade[ball color = black, opacity = .50] (2,2,2) circle (0.125cm); 			
					
					\foreach \x in {0,1,2,3,4}
					\foreach \y in {0,1,2,3,4}
					\foreach \z in {0,1,2,3,4}{
						\ifthenelse{  \lengthtest{\x pt < 4pt}  }{
							\draw[opacity= .20 *  (\z + 1 ) ] (\x,\y,\z) -- (\x+1,\y,\z);
						}{}
						\ifthenelse{  \lengthtest{\y pt < 4pt}  }{
							\draw[opacity= .20 *  (\z + 1 ) ]  (\x,\y,\z) -- (\x,\y+1,\z);
							
						}{}
						\ifthenelse{  \lengthtest{\z pt < 4pt}  }{
							\draw[opacity=.20 * (\z + 1 ) ]  (\x,\y,\z) -- (\x,\y,\z+1);
						}{}
						
					}
					
					%
					
				\end{tikzpicture}
			\end{center}
		\end{subfigure}
		\caption{Example of patch cells in 2D (left) and 3D (right).}
		\label{fig:patches}
	\end{figure}


Associated with each vertex, we define a localized space $V_i \subset H^1_0(\Omega)$ of serendipity elements on each $K \in \mathcal{T}_h$, but vanishing outside of $\Omega^i$:
\begin{equation}
  V_i = \left\{ v \in V_{h,k} : \text{supp}(v) \subseteq \Omega^i \right\}.
\end{equation}
Note that these patch-based subspaces are labeled using subscripts.
	Later, in~\cref{ssec:mg}, we will also introduce function spaces based on a mesh hierarchy, and we shall label those with superscripts instead.
	Examples of vertex patch spaces are shown in~\cref{fig:patches}.  These pictures highlight the considerable reduction in dimensionality that serendipity elements afford.  
	We can also include a comparison to patch sizes for simplicial cells with $\mathcal{P}_k$ elements, since rectangular cells can always be so tesselated.  While $\mathcal{P}_k$ will have fewer degrees of freedom per cell than either $\mathcal{S}_k$ or $\mathcal{Q}_k$, simplicial meshes have many more cells sharing a vertex than do quad/hex meshes. While quad meshes typically have 4 cells sharing a vertex and hex meshes 8, triangular meshes will have 6-7 triangles sharing a vertex and tetrahedral meshes more like 20-30. \cref{patchdofs} shows that, with some representative choices, simplex meshes can offer smaller patch sizes than $\mathcal{Q}_k$ elements but not as small as $\mathcal{S}_k$.  We plot the patch size as a function of degree in \cref{patchdofs}.

	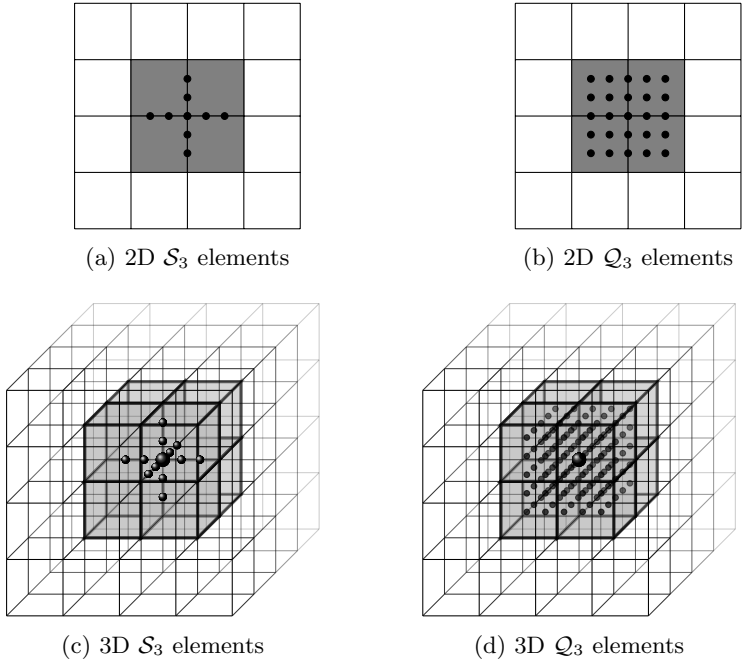
\begin{figure}[htbp]
		\begin{subfigure}[t]{0.45\textwidth}
			\begin{center}
				\begin{tikzpicture}[scale=1.5]
					\draw[step=0.5cm] (-1,-1) grid (1,1);
					\node[rectangle,draw=black, minimum size=0.75cm, fill=gray] at (-0.25,+0.25){};
					\node[rectangle,draw=black, minimum size=0.75cm, fill=gray] at (+0.25,+0.25){};
					\node[rectangle,draw=black, minimum size=0.75cm, fill=gray] at (-0.25,-0.25){};
					\node[rectangle,draw=black, minimum size=0.75cm, fill=gray] at (+0.25,-0.25){};
					\draw[fill] (+0,-0.0) circle (0.03);
					\foreach \x in {-1, 0, 1}{
						\foreach \y in {-1, 0, 1} {
							\pgfmathparse{ abs( abs(\y*\x) - 1) < 0.001 ? int(1) : int(0) }
							\ifnum\pgfmathresult=1 {}
							\else
							{\draw[fill] (\x*0.165,\y*0.165) circle (0.03);
							}
							\fi 	
						}
					}
					
					\foreach \x/\y in {-2/0, 0/2, 0/-2,2/0}{
						\pgfmathparse{ abs( abs(\y*\x) - 1) < 0.001 ? int(1) : int(0) }
						\ifnum\pgfmathresult=1 {}
						\else
						{\draw[fill] (\x*0.165,\y*0.165) circle (0.03);
						}
						\fi 	
						
					}			
				\end{tikzpicture}
			\end{center}
			\caption{2D $\mathcal{S}_3$ elements}
			\label{fig:2dpatchess3}
		\end{subfigure}
		\begin{subfigure}[t]{0.45\textwidth}
			\begin{center}
				\begin{tikzpicture}[scale=1.5]
					\draw[step=0.5cm] (-1,-1) grid (1,1);
					\node[rectangle,draw=black, minimum size=0.75cm, fill=gray] at (-0.25,+0.25){}; 
					\node[rectangle,draw=black, minimum size=0.75cm, fill=gray] at (+0.25,+0.25){}; 
					\node[rectangle,draw=black, minimum size=0.75cm, fill=gray] at (-0.25,-0.25){}; 
					\node[rectangle,draw=black, minimum size=0.75cm, fill=gray] at (+0.25,-0.25){}; 
					\draw[fill] (+0,-0.0) circle (0.03);
					
					\foreach \x in {-2,-1, 0, 1, 2}{
						\foreach \y in {-2,-1, 0, 1, 2} {
							\draw[fill] (\x*0.165,\y*0.165) circle (0.03);  
						}
					}
				\end{tikzpicture}
			\end{center}
			\caption{2D $\mathcal{Q}_3$ elements}
			\label{fig:2dpatchesq3}
		\end{subfigure} \\
		\begin{subfigure}[t]{0.4\textwidth}
			\begin{center}
				\begin{tikzpicture}[scale=0.75,
					thin,cube/.style={very thick,fill ,fill=gray!50}]
					\foreach \x in {1,2}
					\foreach \y in {1,2}
					\foreach \z in {1,2}{
						\filldraw[cube,opacity=.20 *  (\z + 1)  ] (\x,\y,\z) -- (\x,\y+1,\z) -- (\x+1,\y+1,\z) -- (\x+1,\y,\z) -- cycle;
						\filldraw[cube,opacity=.20 *  (\z + 1)  ] (\x,\y,\z) -- (\x,\y+1,\z) -- (\x,\y+1,\z+1) -- (\x,\y,\z+1) -- cycle;
						
						\filldraw[cube,opacity=.20 *  (\z + 1) ]  (\x,\y,\z) --  (\x+1,\y,\z) --  (\x+1,\y,\z+1) --  (\x,\y,\z+1) -- cycle;

						\filldraw[cube,opacity=.20 *  (\z + 1)  ] (\x+1,\y,\z) -- (\x+1,\y+1,\z) -- (\x+1,\y+1,\z+1) -- (\x+1,\y,\z+1) -- cycle;
						
						\filldraw[cube,opacity=.20 *  (\z + 1)  ] (\x,\y+1,\z)  -- (\x+1,\y+1,\z) -- (\x+1,\y+1,\z+1) -- (\x,\y+1,\z+1) -- cycle;
						
						\filldraw[cube,opacity=.20 *  (\z + 1)  ] (\x,\y,\z+1) -- (\x,\y+1,\z+1)  -- (\x+1,\y+1,\z+1) -- (\x+1,\y,\z+1) -- cycle;

					}
					

					\foreach \x in {0,1,2,3,4}
					\foreach \y in {0,1,2,3,4}
					\foreach \z in {0,1,2,3,4}{
						\ifthenelse{  \lengthtest{\x pt < 4pt}  }{
							\draw[opacity= .20 *  (\z + 1 ) ] (\x,\y,\z) -- (\x+1,\y,\z);
						}{}
						\ifthenelse{  \lengthtest{\y pt < 4pt}  }{
							\draw[opacity= .20 *  (\z + 1 ) ]  (\x,\y,\z) -- (\x,\y+1,\z);
							
						}{}
						\ifthenelse{  \lengthtest{\z pt < 4pt}  }{
							\draw[opacity=.20 * (\z + 1 ) ]  (\x,\y,\z) -- (\x,\y,\z+1);
						}{}
						
					}
					
					\shade[ball color = black, opacity=.20 * (2 + 1 ) ] (2,2,2) circle (0.125cm); 
					
					\foreach \i in {-2,-1,1,2}{
						\shade[ball color = black,opacity=.20 * (2 + 1 ) ] (2+0.33*\i,2,2) circle (0.07cm);
						\shade[ball color = black, opacity=.20 * (2 + 1 ) ] (2,2+0.33*\i,2) circle (0.07cm); 
						\shade[ball color = black, opacity=.20 * (2 + 1 + \i) ] (2,2,2+0.33*\i) circle (0.07cm);
					}
					
				\end{tikzpicture}
			\end{center}
			\caption{3D $\mathcal{S}_3$ elements}
		\end{subfigure}
		\begin{subfigure}[t]{0.45\textwidth}
			\begin{center}
				\begin{tikzpicture}[scale=0.75,
					thin,cube/.style={very thick,fill ,fill=gray!50}]
					
					\foreach \x in {1,2}
					\foreach \y in {1,2}
					\foreach \z in {1,2}{
						\filldraw[cube,opacity=.20 *  (\z + 1)  ] (\x,\y,\z) -- (\x,\y+1,\z) -- (\x+1,\y+1,\z) -- (\x+1,\y,\z) -- cycle;
						\filldraw[cube,opacity=.20 *  (\z + 1)  ] (\x,\y,\z) -- (\x,\y+1,\z) -- (\x,\y+1,\z+1) -- (\x,\y,\z+1) -- cycle;
						\filldraw[cube,opacity=.20 *  (\z + 1) ]  (\x,\y,\z) --  (\x+1,\y,\z) --  (\x+1,\y,\z+1) --  (\x,\y,\z+1) -- cycle;
						\filldraw[cube,opacity=.20 *  (\z + 1)  ] (\x+1,\y,\z) -- (\x+1,\y+1,\z) -- (\x+1,\y+1,\z+1) -- (\x+1,\y,\z+1) -- cycle;
						\filldraw[cube,opacity=.20 *  (\z + 1)  ] (\x,\y+1,\z)  -- (\x+1,\y+1,\z) -- (\x+1,\y+1,\z+1) -- (\x,\y+1,\z+1) -- cycle;
						\filldraw[cube,opacity=.20 *  (\z + 1)  ] (\x,\y,\z+1) -- (\x,\y+1,\z+1)  -- (\x+1,\y+1,\z+1) -- (\x+1,\y,\z+1) -- cycle;
					}
					\foreach \x in {0,1,2,3,4}
					\foreach \y in {0,1,2,3,4}
					\foreach \z in {0,1,2,3,4}{
						\ifthenelse{  \lengthtest{\x pt < 4pt}  }{
							\draw[opacity= .20 *  (\z + 1 ) ] (\x,\y,\z) -- (\x+1,\y,\z);
						}{}
						\ifthenelse{  \lengthtest{\y pt < 4pt}  }{
							\draw[opacity= .20 *  (\z + 1 ) ]  (\x,\y,\z) -- (\x,\y+1,\z);
							
						}{}
						\ifthenelse{  \lengthtest{\z pt < 4pt}  }{
							\draw[opacity=.20 * (\z + 1 ) ]  (\x,\y,\z) -- (\x,\y,\z+1);
						}{}
					}
					\foreach \x in {4,5,6,7,8}
					\foreach \y in {4,5,6,7,8}
					\foreach \z in {4,5,6,7,8}{
						\ifthenelse{\x=6 \and \y=6 \and \z=6 }{		
							\shade[ball color = black, opacity=.20 * (\z*0.33 + 1 ) ] (2,2,2) circle (0.125cm); 		}{
							\draw[fill= black,  opacity=.20 * (\z*0.33 + 1 ) ] (\x*0.33,\y*0.33,\z*0.33) circle (0.05cm); } 
					}
				\end{tikzpicture}
			\end{center}
			\caption{3D $\mathcal{Q}_3$ elements}
		\end{subfigure}
		\caption{Patch cells and DoFs for in a typical patches using $\mathcal{S}_3$  and $\mathcal{Q}_3$ elements in 2D are shown in (a) and (b).
			Note that in 2D each $\mathcal{Q}_3$ patch has 25 internal degrees of freedom versus 9 for each $\mathcal{S}_3$ patch.
			Subfigures (c) and (d) show that the reduction is even greater in 3D, going from 125 down to only 13.}
		\label{patches}
	\end{figure}
	
Since $V_i \subset V_{h,k}$ for each $i$, the natural lifting operator $R_i: V_i \rightarrow V_{h,k}$ is trivial, although its adjoint shall appear in defining certain local operators.  We also use the standard norms and inner products available on $V_{h,k}$ on its subspaces as needed.
	
	\begin{figure}[htbp]
		\begin{subfigure}[t]{0.45\textwidth}
			\begin{tikzpicture}[scale=0.65]
				\begin{axis}[
					legend cell align=left,
					legend pos=north west,
					xlabel={Polynomial degree $k$},
					ylabel={DoFs per patch},
					ymode=log, log basis y={10}
					]
					\addplot[dotted,mark=triangle,mark options={solid}] table
					[x=r, y=Pr, col sep=comma] {patchsize_2.csv};
					\addlegendentry{$\mathcal{P}_k$ patch}
					\addplot[dotted,mark=o, mark options={solid}] table
					[x=r, y=Sr, col sep=comma] {patchsize_2.csv};
					\addlegendentry{$\mathcal{S}_k$ patch}
					\addplot[dotted,mark=*,mark options={solid}] table
					[x=r, y=Qr, col sep=comma] {patchsize_2.csv};
					\addlegendentry{$\mathcal{Q}_k$ patch}
				\end{axis}
			\end{tikzpicture}
			\caption{2D patches, with 4 cells and 4 internal edges per quadilateral patch.  For comparison, we include triangular $P_k$ patches assuming 6 cells and edges in each patch.}
			\label{2dpatch}
		\end{subfigure}
		\hfill
		\begin{subfigure}[t]{0.45\textwidth}
			\begin{tikzpicture}[scale=0.65]
				\begin{axis}[
					legend cell align=left,
					legend pos=north west,
					xlabel={Polynomial degree $k$},
					ylabel={DoFs per patch},
					ymode=log, log basis y={10}
					]
					\addplot[dotted,mark=triangle,mark options={solid}] table
					[x=r, y=Pr, col sep=comma] {patchsize_3.csv};
					\addlegendentry{$\mathcal{P}_k$ patch}
					\addplot[dotted,mark=o, mark options={solid} ] table
					[x=r, y=Sr, col sep=comma] {patchsize_3.csv};
					\addlegendentry{$\mathcal{S}_k$ patch}
					\addplot[dotted,mark=*,mark options={solid}] table
					[x=r, y=Qr, col sep=comma] {patchsize_3.csv};
					\addlegendentry{$\mathcal{Q}_k$ patch}
				\end{axis}
			\end{tikzpicture}
			\caption{3D patches with 8
				cells, 12 internal faces, and 6 internal edges per hexahedral
				patch.  For comparison, we include $P_k$ patches with 24 cells, 36 internal faces, and 14 edges per patch.}
			\label{3dpatch}
		\end{subfigure}
		\caption{Number of DoFs in various element patches as a function of polynomial order $k$ after imposing Dirichlet boundary conditions.  The $\mathcal{P}_k$ patches correspond to typical patches in a simplicial mesh, while $\mathcal{Q}_k$ and $\mathcal{S}_k$ correspond to regular patches in a quad/hex mesh.}
		\label{patchdofs}
	\end{figure}
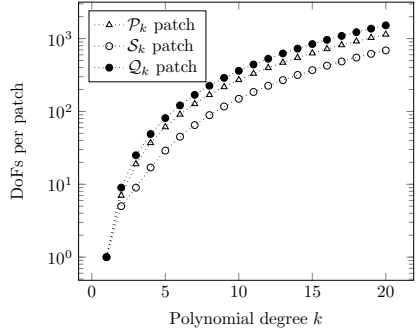
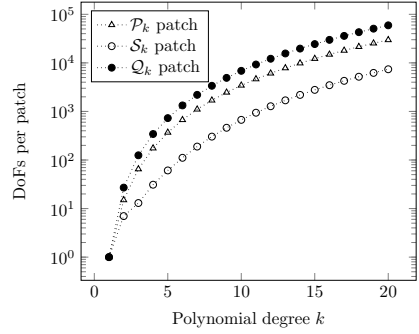
	
	The patch spaces induce a decomposition of the global space:
	\begin{equation}\label{decomp}
		V_{h,k} = \sum_{i=1}^N R_i V_i = \sum_{i=1}^N V_i.
	\end{equation}
	
	On each localized subspace, 
        we define the restricted operators $A_i$ on $V_i$ via a Galerkin approach:
	\begin{equation}\label{restoper2}
		A_i = (R_i)^T A_h R_i,
	\end{equation}
	where $(\cdot)^T$ denotes the adjoint with respect to the $L^2$ inner product.
	
	We now define a preconditioning operator $C_{h,k}$ as follows.   Let $d_{h,k} \in V_{h,k}$ be given.  Then, on each patch, we let $w_i \in V_i$ be the solution to the variational problem
	\begin{equation}
		a_i(w_i, v_i) = (d_{h,k}, v_i), \ \ \ \forall v_i \in V_i.
	\end{equation}
	We then define $w_{h,k} \in V_{h,k}$ by summation:
	\begin{equation}
		w_{h,k} = \sum_{i=1}^N R_i w_i = \sum_{i=1}^N w_i.
	\end{equation}
	Now, we let $C_{h,k}$ be the mapping sending $w_{h,k}$ to $d_{h, k}$.  Its inverse
	is directly computable by solving local variational problems on each vertex patch and summing the liftings of the results.
	Equivalently, we can write
	\begin{equation}\label{asprecond}
		C_{h,k}^{-1} = \sum_{i=1}^N R_i A_i^{-1} (R_i)^T.
	\end{equation}
	
	Among many options in the literature for analyzing such additive Schwarz me\-thods, we follow Sch\"oberl~\cite{Schoeberl1999}.
	We note that our setting -- using patches of the original mesh to deal with high polynomial degree -- naturally leads to a Galerkin formulation with nested subspaces and hence a simpler analysis than required in some contexts.
	
	\begin{theorem}[{\cite[Additive Schwarz Lemma]{Schoeberl1999}}]
		Let us define the splitting norm
		$$
		\vertiii{u_{h,k}}^{2}:=\inf _{u_{h,k}=\sum\limits_{\mathclap{u_{i} \in V_{i}} }R_{i} u_{i}}\sum_{i=1}^{N} \left(A_{i}u_{i}, u_{i}\right)
		$$
		on $V_{h,k}$. It is equal to the norm $\left\|u_{h,k}\right\|_{C_{h,k}}:=\left(C_{h,k} u_{h,k}, u_{h,k}\right)_{h,k}^{1 / 2}$ generated by the additive Schwarz preconditioner (\ref{asprecond}),
		i.e. there holds
		$$
		\vertiii{u_{h,k}} =\left\|u_{h,k}\right\|_{C_{h,k}} \quad \forall u_{h,k} \in V_{h,k}.
		$$
	\end{theorem}
	
	
	
	At several points in our analysis, we will use the Ritz projection~\cite{Thomee, wheeler1973apriori}, which is orthogonal projection with respect to the $H_0^1$ inner product.  
	As we use the projection in a few different contexts, we include the following general result:
	\begin{proposition}\label{prop:ritz}
          Let $a$ be a symmetric coercive bilinear form on $H_0^1(\mathcal{K})$, and $\mathcal{V}, \mathcal{W}$ be subspaces of $H_0^1(\mathcal{K})$.  Let $\Pi$ be the $a$-orthogonal or Ritz projection from $\mathcal{V}$ onto $\mathcal{W}$ with respect to $a$, such that
          \[
          a(\Pi u, v) = a(u, v), \ \ \ v \in \mathcal{W}, \ \ \ u \in \mathcal{V}.
          \]
          Then
	  \begin{enumerate}[label={(\arabic*)},ref={\theproposition~(\arabic*)}]
			\item $\Pi$ is $H^1$ stable, i.e., there is a constant $C_1$ independent of the mesh size and the polynomial degree such that $||\Pi u||_{H^{1}} \leq \text C_1||u||_{H^{1}}$ for all $u$ in $\mathcal{V}$.\label{prop:ritza}
			\item Let $\mathcal{V}$ be a space with mesh size $h$, then $
			\|u-\Pi u\|^2_{L^2} \leq h^{2}C_2 \left\|u\right\|^2_{H^1}$ for all $u$ in $\mathcal{V}$ with $C_2$ independent of the mesh size and the polynomial degree.\label{prop:ritzb}
		\end{enumerate}
	\end{proposition}
	
	In order to verify the conditions for the additive Schwarz theory hold for serendi\-pity spaces, we will adapt some arguments from Pavarino~\cite{Pavarino1993/94} for the tensor-product case. 
	Since $u \mapsto \|\nabla u\|$ is a seminorm on $\mathcal{S}_k$ and $\mathcal{Q}_k$, we define the quotient spaces $\hat{\mathcal{Q}}_{k}=\mathcal{Q}_k / \mathcal{Q}_{0}$ and $\hat{\mathcal{S}}_{k}=\mathcal{S}_{k} /\mathcal{Q}_{0}$ on which it is a norm. 
	Pavarino defines the uniformly bounded interpolation operator for tensor product spaces:
	\begin{equation}\label{tk}
		T_{k}:\hat{\mathcal{Q}}_{k+1}\left([-1,1]^{2}\right)\rightarrow \hat{\mathcal{Q}}_{k}\left([-1,1]^{2}\right),
	\end{equation}
	that interpolates its input at the tensor product of roots of the polynomial
	$
	\mathscr{L}_{k+1}(x)=\int_{-1}^{x} L_{k}(s) d s,
	$
	where $L_{k}(s)$ is the Legendre polynomial of degree $k$.  
	The essential feature is that the operator has $H^1$ stability estimates independent of the polynomial degree.
	While the serendipity space is contained in the tensor-product space, Pavarino's operator $T_k$ is well-defined.  However, for $f \in \hat{\mathcal{S}}_{k+1}$, the resulting $T_k(f)$ need only lie in $\hat{\mathcal{Q}}_k$ and not necessarily in $\hat{\mathcal{S}}_k$.
	In order to obtain spectral bounds for the preconditioner~\eqref{asprecond}, we provide an interpolation operator from $\hat{\mathcal{S}}_{k+1}\left([-1,1]^{2}\right)$ to $  \hat{\mathcal{S}}_{k}\left([-1,1]^{2}\right)$.
	
	\begin{lemma}\label{interlemma} The interpolation operator $\Psi_{k}: \hat{\mathcal{S}}_{k+1}\left([-1,1]^{2}\right) \rightarrow \hat{\mathcal{S}}_{k}\left([-1,1]^{2}\right)$, defined as $\Psi_{k}=\Pi_{k} \circ T_{k} \circ I|_{S_{k+1}}$, where $\Pi_{k}: \hat{\mathcal{Q}}_{k}\left([-1,1]^{2}\right) \rightarrow \hat{\mathcal{S}}_{k}\left([-1,1]^{2}\right)$ is the Ritz projection and $I|_{S_{k+1}}: \hat{\mathcal{S}}_{k+1}\left([-1,1]^{2}\right) \rightarrow \hat{\mathcal{Q}}_{k+1}\left([-1,1]^{2}\right)$ is the natural embedding, is uniformly bounded in the $\|\cdot\|_{H^{1}}$ norm, i.e.
		\begin{equation}
			\left\|\Psi_{k}(f)\right\|_{H^{1}} \leq   \hat{C} \|f\|_{H^{1}}, \quad \forall f \in \hat{\mathcal{S}}_{k+1}\left([-1,1]^{2}\right).
		\end{equation}
	\end{lemma}
	\begin{proof}
		Let $\hat{u}=\left.T_{k} \circ I\right|_{S_{k+1}}(u)$ with $u \in \hat{\mathcal{S}}_{k+1}\left([-1,1]^{2}\right) $. Since $\hat{u} \in \hat{\mathcal{Q}}_{k}\left([-1,1]^{2}\right)$, then $\hat{u} \in H^1_{0}\left([-1,1]^{2}\right)$.
		
		Since the Ritz projecton is $H^1$ stable in $H^1_{0}\left([-1,1]^{2}\right)$ (Proposition~\ref{prop:ritza}), for  $\Pi_{k}: \hat{\mathcal{Q}}_{k}\left([-1,1]^{2}\right) \rightarrow \hat{\mathcal{S}}_{k}\left([-1,1]^{2}\right)$ we have, 
		$$
		||\Pi_{k}\hat{u}||_{H^{1}} \leq C_1||\hat{u}||_{H^{1}}.
		$$
		
		Furthermore, for $\bar{u}=I|_{S_{k+1}}(u)$, notice that, $\bar{u} \in \hat{\mathcal{Q}}_{k+1}\left([-1,1]^{2}\right)$. Then, since the interpolation operator $T_k$ is uniformly bounded (see proof in \cite{Pavarino1993/94})
		$$
		\left\|T_{k}(\bar{u})\right\|_{H^{1}} \leq  C_2 \|\bar{u}\|_{H^{1}}.
		$$ 
		
		Thus,
		$$
		||\Psi_{k}(u)||_{H^{1}} = ||\Pi_{k}\hat{u}||_{H^{1}} \leq C_1||\hat{u}||_{H^{1}} = C_1||T_{k}(\bar{u})||_{H^{1}} \leq C_1 \cdot  C_2 \|\bar{u}\|_{H^{1}}= \hat{C}||u||_{H^{1}},
		$$
		for all $u\in \hat{\mathcal{S}}_{k+1}$.
	\end{proof}
	
	
	From now on we will write $A \leq B$ where $A$ and $B$ are two positive definite matrices whenever $B-A$ is positive semi-definite, and $a \preceq b$ will mean that there exists a constant $c$ independent of $a$, $b$, the discretization parameter $H$, and the polynomial degree $k$, such that $a \leq c b$. Moreover, moving forward we consider that all the constant are independent of the mesh size and the polynomial degree unless otherwise stated.
	\begin{theorem}\label{thm:bounds} The additive Schwarz preconditioner $C_{h,k}$ defined in~\eqref{asprecond} satisfies the spectral bounds
		$H^{2} C_{h,k} \preceq A_{h} \preceq C_{h,k} $, where $H=diam(\Omega^{i})$.
	\end{theorem}
	\begin{proof}
		Each point is covered by no more than $2^n$ subdomains $\Omega^i$ where $n = (2,3)$ is the number of dimensions, thence we can apply Lemma 3.2 (Finite Overlap) in \cite{Schoeberl1999} and obtain the upper bound $A_{h} \leq c_2 C_{h,k}$.
		
	In order to obtain the lower bound, for each $u_{h,k} \in V_{h,k}$, we need a representation  $u_{h,k}=\sum_{i=1}^{N} u_{i},$ where $u_{i} \in V_{i}$, such that
		\begin{equation}\label{lioncond}
			\sum_{i=1}^{N} ||u_{i}||_{A_{h}} \leq C_{0}^{2} ||u_{h,k}||_{A_{h}}, \quad \forall u_{h,k} \in V_{h,k}.
		\end{equation} 
		
		
		Using the operator $T_k$~\eqref{tk}, Pavarino~\cite{Pavarino1993/94} proves that indeed there is a representation for $u_{h,k}$ over $\sum_{i=1}^N R_i V_i$ of the form (\ref{lioncond}) for quadrilateral spaces. Following the same guideline presented in his work, we can obtain this representation for a function defined over serendipity spaces instead. For this purpose, we will use the operator $\Psi_{k}: \hat{\mathcal{S}}_{k+1}\left([-1,1]^{2}\right) \rightarrow \hat{\mathcal{S}}_{k}\left([-1,1]^{2}\right)$ obtained in our \cref{interlemma}.
		
		We now take the construction of a partition of unity $\left\{\psi_{i}\right\}, \psi_{i} \in C^{\infty}(\Omega)$ given by Pavarino~\cite{Pavarino1993/94} which is such that $
		\operatorname{supp} \psi_{i} \subseteq \Omega^{i}$ and $\left\|\nabla \psi_i\right\|_{L_{\infty}} \preceq H^{-1}$, and define, $
		u_{i}=\Psi_{k}(\psi_{i} u_{h,k})$.

		
		Since $\Psi_{k}$ is linear
		$$
		\sum_{i=1}^{N} u_{i}=\sum_{i=1}^{N} \Psi_{k}\left(\psi_{i} u_{h,k}\right)=\Psi_{k} (u_{h,k})=u_{h,k},
		$$
		and so
		\begin{equation}
			\begin{aligned}
				\left\|u_{i}\right\|_{A_{h}}^{2} & \preceq\left\|u_{i}\right\|_{H^1}^{2}=\left\|\Psi_{k}\left(\psi_{i} u_{h}\right)\right\|_{H^1}^{2}.
			\end{aligned}
		\end{equation}
		
		By \cref{interlemma} it follows that,
		\begin{equation}\label{tofollow}
			\begin{aligned}
				\left\|\Psi_{k}\left(\psi_{i} u_{h}\right)\right\|_{H^1}^{2}  \preceq\left\|\psi_{i} u_{h}\right\|_{H^1}^{2}
				& \preceq H^{-2 }\left\|u_{h}\right\|_{L^2(\Omega^i)}^{2}+\left\|\nabla u_{h}\right\|_{L^2(\Omega^i)}^{2},
			\end{aligned}
		\end{equation}
		where we also applied \cite[Lemma\,2.6]{Schoeberl1999} to the last inequality. Therefore,
		\begin{equation} \label{intermresult}
			\begin{aligned}
				\sum_{i=1}^{N}\left\|u_{i}\right\|_{A_{h}}^{2} & \preceq \sum_{i=1}^{N}\left[H^{-2}\left\|u_{h}\right\|_{ L^2(\Omega^{i})}^{2}+\left\|\nabla u_{h}\right\|_{ L^2(\Omega^{i})}^{2}\right] \\
				& \preceq H^{-2 }\left\|u_{h}\right\|_{L^2}^{2}+\left\|\nabla u_{h}\right\|_{L^2}^{2} \preceq H^{-2 }\left\|u_{h}\right\|_{H^1}^{2} \preceq H^{-2 }\left\|u_{h}\right\|_{A_{h}}^{2}.
			\end{aligned}
		\end{equation} 
		
		We obtained a representation as (\ref{lioncond}) with $C^2_0=H^{-2}c^{-1}_1$, with $c_1$ independent from $H$. Moreover, by the Additive Schwarz Lemma we have
		\begin{equation}\label{lowerbound2}
			\begin{aligned}
				\left\|u_{h,k}\right\|^2_{C_{h,k}}&=\inf\limits _{u_{h,k}=\sum\limits_{u_{i} \in V_{i} }R_{i} u_{i}} \sum\limits_{i=1}^{M} \left(A_{i}u_{i}, u_{i}\right) \leq \sum\limits_{i=1}^{N}  \left(A_i u_{i},  u_{i}\right)\\&= \sum\limits_{i=1}^{N} \left(R_i^T A_h R_i u_{i},  u_{i}\right)= \sum\limits_{i=1}^{N} ||u_{i}||^2_{A_{h}} \leq H^{2} c_{h,k} ||u_{h,k}||^2_{A_{h}},
			\end{aligned}
		\end{equation}
		where we used \eqref{intermresult} for the last inequality.
	\end{proof}	
	
	As a remark, since $H = diam(\Omega^i),$ we also have $H = 2h$.

	\subsection{Additive Schwarz with a low-order space}\label{subsec:asloworder}
	\cref{thm:bounds} shows that a patch smoother provides conditioning estimates independent of the polynomial degree, but leaves the condition number scaling like $H^{-2} \sim h^{-2}$.
        To also eliminate the dependence on the mesh size, we include a global low-order space in the decomposition~\eqref{decomp}:
	\begin{equation}
		V_{0} = \left\{ v \in V_{h,k} : v|_K \in \mathcal{Q}_1(K), \ \ \ K \in \mathcal{T}_h \right\}.
	\end{equation}
	There is also a natural inclusion operator $R_0: V_0 \rightarrow V_{h,k}$ since $V_0$ is a subspace of $V_{h,k}$.
	
	We can edit our subspace decomposition to include this space, so that
	\begin{equation}
		V_{h,k} = R_0 V_0 + \sum_{i=1}^N R_i V_i = V_0 + \sum_{i=1}^N V_i = \sum_{i=0}^N V_i.
	\end{equation}
	
	$V_0$ is defined on the entire mesh, but it includes only lowest-order functions and so plays the role of a ``coarse grid'' space.  We define  $a_0 : V_0 \times V_0 \rightarrow \mathbb{R}$:
	\begin{equation}
		a_0(u_0, v_0) = a(R_0 u_0, R_0 v_0), 
	\end{equation}
	and associated operator $A_0$ on $V_0$ via a Galerkin approach
	\begin{equation}
		A_0 = (R_0)^T A_h R_0.
	\end{equation}
	
	
	As before, this new decomposition can also be used to define an additive Schwarz preconditioner:
	\begin{equation}\label{twolevelprec}
		C_{h,k}^{-1} = R_0 A_0^{-1} (R_0)^T + \sum_{i=1}^N R_i A_i^{-1} (R_i)^T.
	\end{equation}
	That is, an application of this preconditioner requires solving a local problem on each patch together with solving a global system on the $\mathcal{Q}_1$ subspace. 
	
	\begin{lemma} \cite[Optimal two level preconditioner]{Schoeberl1999}\label{opttwolevprec}
		Assume the following statements are true:
		\begin{enumerate}[label={(\arabic*)},ref={\theproposition~(\arabic*)}]
			\item The overlap of local spaces is bounded by $N_O$.
			\item There exists a continuous interpolation operator $I_{0}: V_{h,k} \rightarrow V_{0},$ i.e.
			\begin{equation} 
				\left\|I_{0} u_{h,k}\right\|_{A_{h}} \leq c_{I}\left\|u_{h,k}\right\|_{A_{h}} \quad \forall u_{h,k} \in V_{h,k}
			\end{equation}
			\item The local splitting of the difference $u_{f}:=u_{h}- I_{0} u_{h}$ is stable, i. $e$.
			\begin{equation}\label{condiv}
				\inf\limits_{u_{f}=\sum{u_{i}} \atop {u_{i} \in V_{i}\atop i=1,...,N} } \sum\left\|u_{i}\right\|_{A_h}^{2}  \leq c_{L}\left\|u_{h}\right\|_{A_{h}}^{2}
			\end{equation}
		\end{enumerate}
		Then the two level preconditioner fulfills the optimal spectral bounds
		$$
		c_{1} C_{h,k} \leq A_{h} \leq c_{2} C_{h,k}
		$$
		with
		$$
		\begin{aligned}
			c_{1} &:=\left(c_{I}^{2}+c_{L}\right)^{-1} \\
			c_{2} &:=\left(1+N_{O}\right)
		\end{aligned}
		$$
	\end{lemma}  
	
	\begin{theorem}\label{thm:twolevelpreocond}
		The two level preconditioner $C_{h,k}$~\eqref{twolevelprec} satisfies optimal spectral bounds:  $C_{h,k} \preceq A_{h} \preceq C_{h,k}$, with both bounds independent of $h$ and $k$.
	\end{theorem} 
	
	\begin{proof}
		We just  need to verify the conditions given in \cref{opttwolevprec}. 
		
		As we mentioned before, the number of overlapping cells is bounded by $4$ in 2D and by $8$ in 3D. Moreover, since we have $V_0 \subset V_{h,k}$, the continuity of the lifting/natural inclusion operator $R_0$, is trivial. 
		
		In addition, let $I_0: V_{h,k} \rightarrow V_0$ be the Ritz projection with respect to the bilinear form $a(\cdot,\cdot)$, so that
		$$
		a\left( I_{0} u, v\right)=a(u, v), \quad \forall v \in V_0, \quad \text { for } u \in V_{h,k}.
		$$
		Once again, we will use it only to prove the existence of bounds, and we do not propose to use it during computations.
		
		The properties of $I_0$ and the bilinear form give that
		\begin{align*}
			\|I_0 u_{h,k}\|^2_{A_{h}}=a(I_0 u_{h,k},I_0 u_{h,k})=a(u_{h,k},I_0 u_{h,k}) \leq \|u_{h,k}\|_{A_{h}} \|I_0 u_{h,k}\|_{A_{h}},
		\end{align*}
		so that
		$$
		\left\|I_{0} u_{h,k}\right\|_{A_{h}} \leq \left\|u_{h,k}\right\|_{A_{h}} \quad \forall u_{h,k} \in V_{h,k}.
		$$
		Hence, $I_0$ is bounded in the energy norm independent of the polynomial degree or mesh parameter of $V_{h,k}$.
		
		
		
		
		Furthermore,
		\begin{equation}
			\begin{aligned}
				\left\|u_{f}\right\|^2_{L^2} &=\left\|u_{h,k}- I_{0} u_{h,k}\right\|^2_{L^2} \preceq H^{2}\left\|u_{h,k}\right\|^2_{H^1},
			\end{aligned}
		\end{equation}
		where the last inequality comes from Proposition~\ref{prop:ritzb}. Then, we follow,
		\begin{equation}
			\begin{aligned}
				\left\|u_{f}\right\|^2_{H^1} &=\left\|u_{h,k}-I_{0} u_{h,k}\right\|^2_{H^1} 
				\leq 2(\left\|u_{h,k}\right\|^2_{H^1}+\left\| I_{0} u_{h,k}\right\|^2_{H^1} )
				\preceq \left\|u_{h,k}\right\|^2_{H^1},
			\end{aligned}
		\end{equation}
		where in the last step we used the $H^1$ stability of the Ritz projection as given in Proposition~\ref{prop:ritza}.
		
		Now we define $u_{f,i}=\Psi_{k}(\psi_{i} u_{f}) $ and proceeding as in the proof of the lower bound in \cref{thm:bounds} we get
		\begin{equation*}
			\begin{aligned}
				\sum_{i=1}^{N}\left\|u_{f,i}\right\|_{A_{h}}^{2}& \preceq H^{-2 }\left\|u_{f}\right\|_{L^2}^{2}+\left\|\nabla u_{f}\right\|_{L^2}^{2} 
				\preceq H^{-2 } H^2 \left\|u_{h,k}\right\|_{H^1}^{2}+\left\| u_{f}\right\|_{H^1}^{2}
				\preceq \left\|u_{h,k}\right\|_{H^1}^{2}.\end{aligned}
		\end{equation*} 
		
		Finally, we can say
		\begin{equation}\label{pavarinolowerbound}
			\inf\limits_{u_{f}=\sum{u_{i}} \atop {u_{i} \in V_{i}\atop i=1,...,N} } \sum\left\|u_{i}\right\|_{A_h}^{2} \leq \sum\limits^N_{i=1}{||u_{f,i}||_{A_h}^{2}} \preceq \left\|u_{h,k}\right\|_{A_h}^{2}.
		\end{equation}

	\end{proof}
	\subsection{Multigrid algorithms with patch smoothers}\label{ssec:mg}
	Including a global but low-order approximating space in the additive Schwarz decomposition is one approach to obtaining an optimal-order algorithm.
	On the other hand, one may define a coarse space by geometrically coarsening the finite element mesh without changing the polynomial degree of the space -- a multigrid method.
	In this section, we prove that applying the additive Schwarz patch smoother~\eqref{asprecond} on each level of a V-cycle provides a contraction factor independent of the mesh size and polynomial degree.
	By contrast, such estimates do not hold for classical pointwise smoothers such as damped Jacobi.
	Here, we follow the approach~\cite{ArnoldFalkWinther}, which applies the same abstract setting to the different use case of patch-smooted problems in $H(\text{div})$.
	
	%
	Abstractly, we posit a sequence of nested spaces (in practice, obtained through a mesh hierarchy) given by
	$$
	V^{1} \subset V^{2} \subset \ldots \subset V^{L}=V_{k,h},
	$$
	recalling that we are using superscripts here rather than the subscripts in our patch-based space decompositions.
	
	
	Define $A^{l}: V^{l} \rightarrow V^{l}$ by
	$$
	\left(A^{l} v, w\right)= a(v, w) \text { for all } v, w \in V^{l}.
	$$
	Now, to prove the optimality of solvers using multigrid methods we shall apply the following theorem introduced in \cite{ArnoldFalkWinther}. Let $\Pi^{l}: V^{L} \rightarrow V^{l}$ be the orthogonal/Ritz projection with respect to the bilinear form $a$. Also, suppose that we are given for each $l>1$ a linear operator $D^{l}: V^{l} \rightarrow V^{l}$, which is the ``smoother'' and is intended to behave like an approximation to $(A^{l})^{-1}$. Finally, we use the standard $\mathrm{V}$-cycle multigrid algorithm by applying the smoother $D^{l}$ as defined in \cite{ArnoldFalkWinther} with operators $\Theta^{l}: V^{l} \rightarrow V^{l}$ beginning with $\Theta^{1}=(A^{1})^{-1}$. Given these conditions the theorem is stated as follows.  
	
	\begin{theorem}\label{thm:mgtheorem} Suppose that for each level $2 \leq l \leq L$ we have a  preconditioner $D^{l}: V^{l} \rightarrow V^{l}$, which is scaled such that $A^{l} \leq D^{l}$, is symmetric with respect to the $L^2$ inner product and positive semidefinite, and satisfies the conditions
		\begin{align}
			A\left(\left[I-(D^{l}) A^{l}\right] v, v\right) \geq 0 \text { for all } v \in V^{l},
		\end{align}
		and
		\begin{align}
			\left((D^{l})^{-1}\left[I-\Pi^{l-1}\right] v,\left[I-\Pi^{l-1}\right] v\right) \leq \alpha A\left(\left[I-\Pi^{l-1}\right] v,\left[I-\Pi^{l-1}\right] v\right) \text { for all } v \in V^{l},
		\end{align}
		where $\alpha$ is some constant. Then
		$$
		0 \leq A\left(\left[I-\Theta^{l} A^{l}\right] v, v\right) \leq \delta A(v, v) \text { for all } v \in V^{l}
		$$
		where $\delta=\alpha /(\alpha+2 m)$. Moreover, the error operator $I-\Theta^{L} A^{L}$ is a positive definite contraction on $V^{L}$ whose operator norm relative to the $A$ inner product is bounded by $\delta$. Also, the eigenvalues of $\Theta^{L} A^{L}$ belong to the interval $[1-\delta, 1]$.
	\end{theorem}

        We can use a patch smoother of the form~\eqref{asprecond} on each level $l$ (after suitable scaling) to obtain degree-independent multigrid estimates.
	We decompose each $V^l$ into patches as in \cref{patchsmoothers}, which leads to
	\begin{equation}
		V^l_{h,k} = \sum_{i=1}^{N_l} R^{l}_{i} V^l_i= \sum_{i=1}^N V^l_i,
	\end{equation}
	and then construct preconditioner \eqref{asprecond} resulting in,
	%
	\begin{equation}\label{mgprecond}
		D^{l} = C_{h_l,k_l}^{-1} = \sum_{i=1}^{N_l} R^{l}_{i} (A^{l}_{i})^{-1} (R^{l}_{i})^T,
	\end{equation}
	where $R^{l}_{i}: V^{l}_{i} \rightarrow V^{l}_{h,k}$ is the  lifting operator at level $l$. Furthermore, notice that
	$$
	\Pi^{l}_{i}=R^{l}_{i} (A^{l}_{i})^{-1} (R^{l}_{i})^{T} A
	$$
	is the $A(., .)$-orthogonal or Ritz projection to $R^l_{i} V^l_{i}$ \cite{Schoeberl1999}. Thus, we can rewrite \eqref{mgprecond} as proposed in \cite{ArnoldFalkWinther},
	\begin{equation}\label{mgprecondarnold}
		D^{l} = \eta \sum_{i=1}^{N_l} \Pi^{l}_{i} (A^{l})^{-1},
	\end{equation}
	where $\eta$ is the scaling factor.
	
	We can also define $\Pi^{l}: V^L \rightarrow V^l $
	$$
	\Pi^{l}= \sum^{N_l}_{i=1} \Pi^{l}_i.
	$$

        \begin{theorem}
          The hypotheses of \cref{thm:mgtheorem} are satisfied by the patch smoother \eqref{mgprecond}. 
        \end{theorem}
	\begin{proof}
	  The proof for the first hypothesis is exactly the one given in \cite{ArnoldFalkWinther}, and so is omitted.
		
		
		The second hypothesis reduces to showing that for $v=\left(I-\Pi^{l-1}\right) u, u \in V^{l}$, we can decompose $v$ as $\sum^{N_l}_{i} v_{i}$ with $v_{i} \in V^{l}_{i}$ such that
		\begin{align}
			\sum^{N_l}_{i=1} A\left(v_{i}, v_{i}\right) \leq c A(v, v).
		\end{align}
		Since we are reducing the analysis to the nested case this follows immediately from \eqref{pavarinolowerbound}. Notice that all we need to prove is that $v=\left(I-\Pi^{l-1}\right) u$ is in $V^{l}$. We know that $u \in V^{l}$, thus $Iu\in V^{l}$ and $\Pi^{l-1}u \in V^{l-1} \subset V^{l}$, so $v \in V^{l}$. Decomposing $v_i=\Psi_{k}(\psi_{i} v)$ and proceeding as in the proof of the lower bound in \cref{thm:twolevelpreocond} we get 
		\begin{equation}
			\sum\limits^{N_l}_{i=1}{||v_{i}||_{A}^{2}} \preceq \left\|v\right\|_{A}^{2}.
		\end{equation}
		
		Therefore, the hypotheses of \cref{thm:mgtheorem} are satisfied.
	\end{proof}

	\section{Numerical Results}\label{sec:num}
	We have conducted a simple set of numerical expe\-riments illustrating our theory using Firedrake~\cite{Rathgeber:2016}, an automated system for finite element methods.  Recent work~\cite{crum2022bringing} in the Firedrake project has enabled a wide class of serendipity elements.  We use these elements plus many solver features available through the rich interface to PETSc~\cite{kirby2018solver} to conduct our investigation.  This includes a kind of two-way interface via petsc4py allowing PETSc extensions to import Firedrake, and also includes an integration of FIredrake with PETSc's multigrid functio\-nality~\cite{kirby2018solver,mitchell2016high}.
	
	We test our methods on the Poisson problem in two and three dimensions and planar elasticity.
	Our experiments were run on a Dell Precision Workstation with dual 14-core Xeon E5-2697 CPUs running at 2.60GHz and 256GB of RAM.
	The full Firedrake code stack supports distributed memory parallelism.
	Since the software components for forming and solving patch problems are designed for parallelism~\cite{farrell2019pcpatch}, and Firedrake-based implementations using patch smoothers for other problems are known to scale well across many nodes~\cite{abu2022monolithic, farrell2019augmented}.
	Therefore, we focus on the algorithmic performance of the additive Schwarz and multigrid methods rather than parallel scaling issues.  However, we do use multiple workstation cores to accelerate our computations.
	
	Firedrake's solver infrastructure has expedited the implementation of our me\-thods.
	In each case we use conjugate gradient iteration with a relative stopping tole\-rance of $10^{-12}$ with our various preconditioners.
	The two-level Schwarz method~\eqref{twolevelprec} was applied via Firedrake's \lstinline{P1PC} preconditioner.
	This class uses the Firedrake bilinear form to generate code for building the small patch problems and the low-order global problem and then sets up user-configurable solvers to be used for each, which are then solved during run-time.
	We solve the local patch problems using Firedrake's \lstinline{ASMStarPC}, which uses  connectivity information from the underlying PETSc DMPLEX~\cite{lange2016efficient} to algebraically extract degrees of freedom for the patches and interface to PETSc's existing additive Schwarz infrastructure.  Because Firedrake does not support general hexahedral meshes, only extruded quadrilateral meshes, we had to adapt this class to identify the patches for our three-dimensional calculations.
	Various matrix formats (dense or sparse) and PETSc solvers are available for the subproblems, and we have found that direct factorization on dense matrices offers the best performance for the patches.
	On the global low-order space we used a fairly standard geometric multigrid algorithm -- two iterations of Jacobi with Chebyshev acceleration on each level and a sparse direct method (MUMPS) on the coarsest mesh.
        We can also apply the multigrid functionality direction the high-order discretization, using two sweeps of the \lstinline{ASMStarPC} on each level with a direct solver on the coarse grid, as analyzed in~\cref{ssec:mg}.
	
	\begin{figure}
		\begin{center}
			\input{square_scatter}
		\end{center}
		\caption{Coarse unstructured mesh with 96 vertices and 120 quadrilaterals.}
		\label{fig:quadmsh}
	\end{figure}
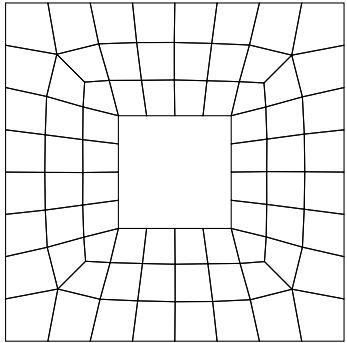
	
	For the Poisson problem in 2D, we chose $\Omega$ to be a $3 \times 3$ square with a $1.5 \times 1.5$ square excluded from the center, as shown in \cref{fig:quadmsh}, and we perform numerical experiments on uniform refinements of this.
	We chose Dirichlet boundary conditions and forcing function $f(x, y)$ such that the true solution is $u(x, y) = e^{xy} \sin(3 \pi x) \sin(4 \pi y)$.
	We approximated the problem with both $\mathcal{S}_k$ and $\mathcal{Q}_k$ elements with $k=2,3,4$ on uniform refinements of the base mesh.  We measured the $L^2$ error of each result, and observed optimal-order convergence rates for both serendipity and tensor-product elements.

	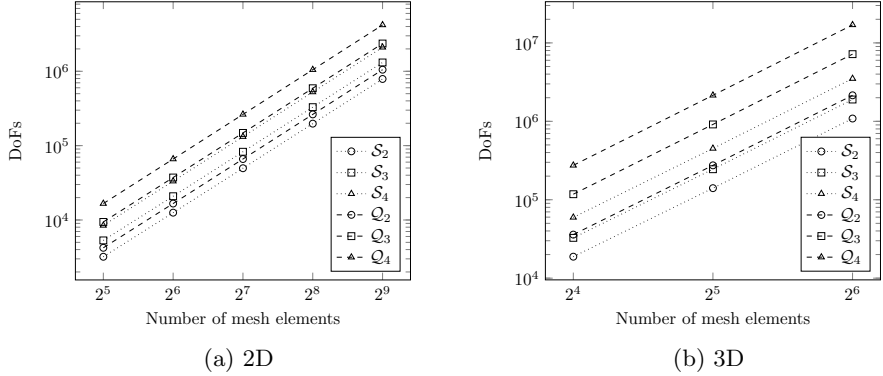
\begin{figure}
          \begin{center}
		\begin{subfigure}[l]{0.48\textwidth}
			\begin{tikzpicture}[scale=0.65]
				\begin{axis}[xtick={32, 64, 128, 256, 512},
					ylabel=DoFs,
					xlabel=Number of mesh elements,
					xmode=log,
					ymode=log,
					log basis x={2},
					log basis y={10},
					legend pos = south east]
					\addplot[dotted,mark=o, mark options={solid}]
					coordinates {(32, 3201) (64, 12545) (128, 49665)
						(256, 197633) (512, 788481)};
					\addlegendentry{$\mathcal{S}_2$};
					\addplot[dotted,mark=square, mark options={solid}]
					coordinates {(32, 5313) (64, 20865) (128, 82689) (256, 329217) (512, 1313793)};
					\addlegendentry{$\mathcal{S}_3$};
					\addplot[dotted,mark=triangle, mark options={solid}]
					coordinates {(32,8449 ) (64, 33281) (128, 132097) (256, 526337) (512, 2101249)};
					\addlegendentry{$\mathcal{S}_4$};
					\addplot[dashed,mark=o, mark options={solid}]
					coordinates {(32, 4225) (64, 16641) (128, 66049)
						(256, 263169) (512, 1050625)};
					\addlegendentry{$\mathcal{Q}_2$};
					\addplot[dashed,mark=square, mark options={solid}]
					coordinates {(32, 9409) (64, 37249) (128, 148225)
						(256, 591361) (512, 2362369)};
					\addlegendentry{$\mathcal{Q}_3$};
					\addplot[dashed,mark=triangle, mark options={solid}]
					coordinates {(32, 16641) (64, 66049) (128, 263169)
						(256, 1050625) (512, 4198401)};
					\addlegendentry{$\mathcal{Q}_4$};
				\end{axis}
			\end{tikzpicture}
			\caption{2D}
			\label{fig:2dpoissondofs}
		\end{subfigure}
	\begin{subfigure}[l]{0.48\textwidth}
			\begin{tikzpicture}[scale=0.65]
				\begin{axis}[xtick={16, 32, 64},
					ylabel=DoFs,
					xlabel=Number of mesh elements,
					xmode=log,
					ymode=log,
					log basis x={2},
					log basis y={10},
					legend pos = south east]
					\addplot[dotted,mark=o, mark options={solid}]
					coordinates {(16, 18785) (32, 140481) (64, 1085825)};
					\addlegendentry{$\mathcal{S}_2$};
					\addplot[dotted,mark=square, mark options={solid}]
					coordinates {(16, 32657) (32, 245025) (64, 1897025)};
					\addlegendentry{$\mathcal{S}_3$};
					\addplot[dotted,mark=triangle, mark options={solid}]
					coordinates {(16, 59585) (32, 450945) (64, 3506945)};
					\addlegendentry{$\mathcal{S}_4$};
					\addplot[dashed,mark=o, mark options={solid}]
					coordinates {(16, 35937) (32, 274625) (64, 2146689)};
					\addlegendentry{$\mathcal{Q}_2$};
					\addplot[dashed,mark=square, mark options={solid}]
					coordinates {(16, 117649) (32, 912673) (64, 7189057)};
					\addlegendentry{$\mathcal{Q}_3$};
					\addplot[dashed,mark=triangle, mark options={solid}]
					coordinates {(16, 274625) (32, 2146689) (64, 16974593)};
					\addlegendentry{$\mathcal{Q}_4$};
				\end{axis}
			\end{tikzpicture}
			\caption{3D}
			\label{fig:3dpoissondofs}
	\end{subfigure}
        \end{center}
	\caption{Number of degrees of freedom in $\mathcal{S}_k$ and $\mathcal{Q}_k$ on regular meshes.}
	\end{figure}
	
	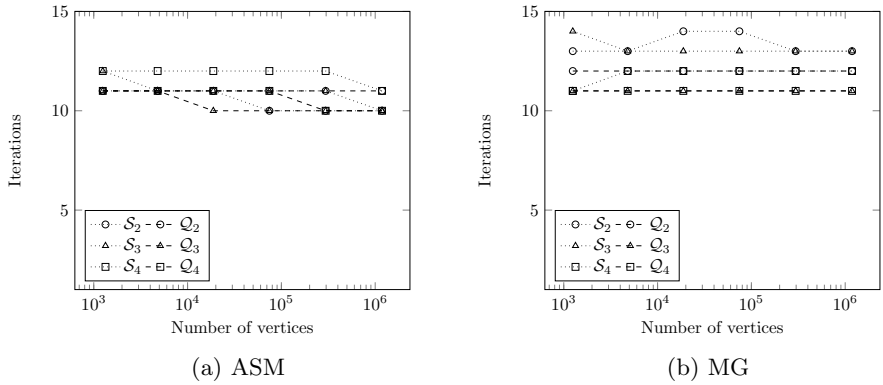
\begin{figure}
		\begin{subfigure}[l]{0.48\textwidth}
			\begin{tikzpicture}[scale=0.65]
				\begin{axis}[xlabel={Number of vertices},
					ylabel=Iterations,
					xmode=log,
					log basis x={10}, ymin=1, ymax=15,
					legend pos = south west,
					legend style = {legend columns = 2}]
					\addplot[dotted,mark=o, mark options={solid}]
					table [x=NV, y=ITS, col sep=comma]{code/results_2d_s2_p1pc.csv};
					\addlegendentry{$\mathcal{S}_2$};
					\addplot[dashed,mark=o, mark options={solid}]
					table [x=NV, y=ITS, col sep=comma]{code/results_2d_q2_p1pc.csv};
					\addlegendentry{$\mathcal{Q}_2$};
					\addplot[dotted,mark=triangle, mark options={solid}]
					table [x=NV, y=ITS, col sep=comma]{code/results_2d_s3_p1pc.csv};
					\addlegendentry{$\mathcal{S}_3$};
					\addplot[dashed,mark=triangle, mark options={solid}]
					table [x=NV, y=ITS, col sep=comma]{code/results_2d_q3_p1pc.csv};
					\addlegendentry{$\mathcal{Q}_3$};
					\addplot[dotted,mark=square, mark options={solid}]
					table [x=NV, y=ITS, col sep=comma]{code/results_2d_s4_p1pc.csv};
					\addlegendentry{$\mathcal{S}_4$};
					\addplot[dashed,mark=square, mark options={solid}]
					table [x=NV, y=ITS, col sep=comma]{code/results_2d_q4_p1pc.csv};
					\addlegendentry{$\mathcal{Q}_4$};				
				\end{axis}
			\end{tikzpicture}
			\label{ASMPoisson2dits}
			\caption{ASM}
		\end{subfigure}
		\begin{subfigure}[l]{0.48\textwidth}
			\begin{tikzpicture}[scale=0.65]
				\begin{axis}[xlabel={Number of vertices},
					ylabel=Iterations,
					xmode=log,
					log basis x={10}, ymin=1, ymax=15,
					legend pos = south west,
					legend style = {legend columns = 2}]
					\addplot[dotted,mark=o, mark options={solid}]
					table [x=NV, y=ITS, col sep=comma]{code/results_2d_s2_mg.csv};
					\addlegendentry{$\mathcal{S}_2$};
					\addplot[dashed,mark=o, mark options={solid}]
					table [x=NV, y=ITS, col sep=comma]{code/results_2d_q2_mg.csv};
					\addlegendentry{$\mathcal{Q}_2$};
					\addplot[dotted,mark=triangle, mark options={solid}]
					table [x=NV, y=ITS, col sep=comma]{code/results_2d_s3_mg.csv};
					\addlegendentry{$\mathcal{S}_3$};
					\addplot[dashed,mark=triangle, mark options={solid}]
					table [x=NV, y=ITS, col sep=comma]{code/results_2d_q3_mg.csv};
					\addlegendentry{$\mathcal{Q}_3$};
					\addplot[dotted,mark=square, mark options={solid}]
					table [x=NV, y=ITS, col sep=comma]{code/results_2d_s4_mg.csv};
					\addlegendentry{$\mathcal{S}_4$};
					\addplot[dashed,mark=square, mark options={solid}]
					table [x=NV, y=ITS, col sep=comma]{code/results_2d_q4_mg.csv};
					\addlegendentry{$\mathcal{Q}_4$};
				\end{axis}
			\end{tikzpicture}
			\label{MGPoisson2dits}
			\caption{MG}
		\end{subfigure}	
		\caption{Iteration counts for ASM and MG-preconditioned conjugate gradients for the 2D Poisson problem.}
		\label{Poisson2dits}
	\end{figure}
	
	\cref{Poisson2dits} plots the iteration counts obtained using additive Schwarz and multigrid preconditioners for both element families.
	Both methods vary only slightly among meshes and discretization spaces, consistent with the theory we have developed.

        In order to assess our timing results, consider \cref{fig:2dpoissondofs}, which plots the size of each approximating space on a regular $N \times N$ mesh (the same trends hold for our unstructured mesh).
  We see the ordering
	\[
	\dim \mathcal{S}_2 <
	\dim \mathcal{Q}_2 <
	\dim \mathcal{S}_3 <
	\dim \mathcal{S}_4 <
	\dim \mathcal{Q}_3 <
	\dim \mathcal{Q}_4.
	\]
	and $\mathcal{S}_4$ and $\mathcal{Q}_3$ are particularly close.  Consequently, we might hope for 2D serendipity elements to provide a higher order of accuracy than tensor-product elements for a similar run-time.  
	\cref{fig:2dpoissontime} plots the time needed to construct the preconditioners and solve the resulting linear system using 16 workstation cores.
	\cref{2dasmsetup} and \cref{2dmgsetup} reveal a nontrivial $\mathcal{O}(1)$ overhead, notably worse for $\mathcal{S}_k$ elements than $\mathcal{Q}_k$.
	Because the $\mathcal{Q}_k$ bases are interpolatory at given points, Firedrake is able to optimize transfer operations, but there may be additional issues as well.  At any rate, after about $10^4$ vertices (divided over 16 cores), the preconditioner setup and system solve scale well with the number of vertices.  
	
	The ASM-preconditioned solver seems to demonstrate reduced run-time for $\mathcal{S}_k$ over $\mathcal{Q}_k$ elements, in line with our hopes, at least asymptotically.  Considering just the solve phase, we see that $\mathcal{S}_2$ gives a modest win over $\mathcal{Q}_2$, and the gain is larger at higher degree.  Moreover, the solve times for $\mathcal{S}_3$ and $\mathcal{Q}_2$, are comparable, as are $\mathcal{S}_4$ and $\mathcal{Q}_3$.  So for similar run time, serendipity elements with ASM preconditioners seem to deliver one higher order of accuracy than tensor-product elements.
        
	For multigrid, however, the picture is somewhat ambiguous.
	On meshes we consider, multigrid for $\mathcal{Q}_k$ elements outperforms that for $\mathcal{S}_k$ spaces for degrees 2 and 3.
	This seems to be due to fast grid transfer operations available for $\mathcal{Q}_k$ elements compensating for the larger patch spaces.
	However, in all of these cases, the ASM-preconditioned iteration seems to give better overall performance than multigrid and indeed reflects the kinds of benefits we hope to obtain from serendipity spaces.

	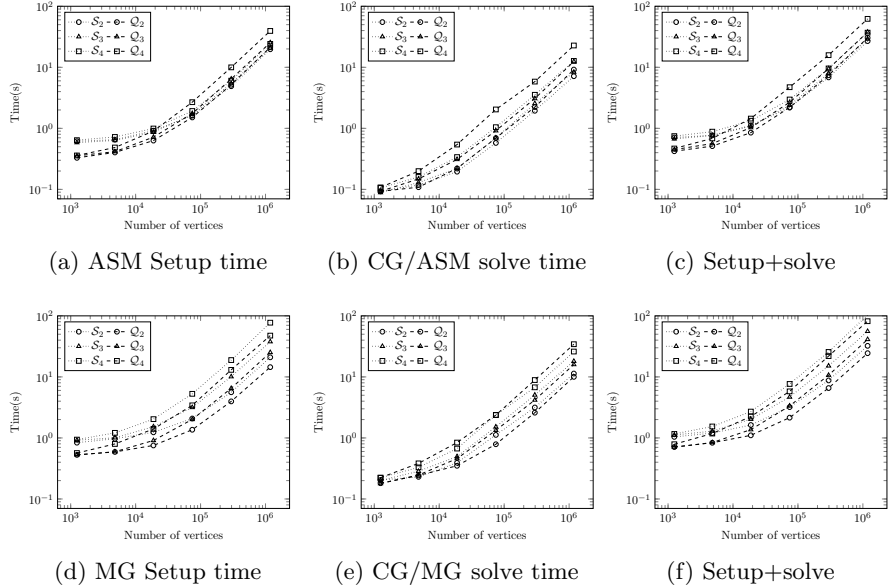
\begin{figure}
		\begin{subfigure}[l]{0.31\textwidth}
			\begin{tikzpicture}[scale=0.45]
				\begin{axis}[xlabel={Number of vertices},
					ylabel=Time(s),
					xmode=log, ymode=log,
					log basis x={10},
					ymin=7e-2, ymax=1e2,
					legend pos = north west,
					legend style = {legend columns = 2}]
					\addplot[dotted,mark=o, mark options={solid}]
					table [x=NV, y=PCSETUP, col sep=comma]{code/results_2d_s2_p1pc.csv};
					\addlegendentry{$\mathcal{S}_2$};
					\addplot[dashed,mark=o, mark options={solid}]
					table [x=NV, y=PCSETUP, col sep=comma]{code/results_2d_q2_p1pc.csv};
					\addlegendentry{$\mathcal{Q}_2$};
					\addplot[dotted,mark=triangle, mark options={solid}]
					table [x=NV, y=PCSETUP, col sep=comma]{code/results_2d_s3_p1pc.csv};
					\addlegendentry{$\mathcal{S}_3$};
					\addplot[dashed,mark=triangle, mark options={solid}]
					table [x=NV, y=PCSETUP, col sep=comma]{code/results_2d_q3_p1pc.csv};
					\addlegendentry{$\mathcal{Q}_3$};
					\addplot[dotted,mark=square, mark options={solid}]
					table [x=NV, y=PCSETUP, col sep=comma]{code/results_2d_s4_p1pc.csv};
					\addlegendentry{$\mathcal{S}_4$};
					\addplot[dashed,mark=square, mark options={solid}]
					table [x=NV, y=PCSETUP, col sep=comma]{code/results_2d_q4_p1pc.csv};
					\addlegendentry{$\mathcal{Q}_4$};
				\end{axis}
			\end{tikzpicture}
			\caption{ASM Setup time}
			\label{2dasmsetup}
		\end{subfigure}
		\begin{subfigure}[c]{0.3\textwidth}
			\begin{tikzpicture}[scale=0.45]
				\begin{axis}[xlabel={Number of vertices},
					ylabel=Time(s),
					xmode=log, ymode=log,
					log basis x={10},
					ymin=7e-2, ymax=1e2,          
					legend pos = north west,
					legend style = {legend columns = 2}]
					\addplot[dotted,mark=o, mark options={solid}]
					table [x=NV, y=KSPSOLVE, col sep=comma]{code/results_2d_s2_p1pc.csv};
					\addlegendentry{$\mathcal{S}_2$};
					\addplot[dashed,mark=o, mark options={solid}]
					table [x=NV, y=KSPSOLVE, col sep=comma]{code/results_2d_q2_p1pc.csv};
					\addlegendentry{$\mathcal{Q}_2$};
					\addplot[dotted,mark=triangle, mark options={solid}]
					table [x=NV, y=KSPSOLVE, col sep=comma]{code/results_2d_s3_p1pc.csv};
					\addlegendentry{$\mathcal{S}_3$};
					\addplot[dashed,mark=triangle, mark options={solid}]
					table [x=NV, y=KSPSOLVE, col sep=comma]{code/results_2d_q3_p1pc.csv};
					\addlegendentry{$\mathcal{Q}_3$};
					\addplot[dotted,mark=square, mark options={solid}]
					table [x=NV, y =KSPSOLVE, col sep=comma]{code/results_2d_s4_p1pc.csv};
					\addlegendentry{$\mathcal{S}_4$};
					\addplot[dashed,mark=square, mark options={solid}]
					table [x=NV, y=KSPSOLVE, col sep=comma]{code/results_2d_q4_p1pc.csv};
					\addlegendentry{$\mathcal{Q}_4$};
				\end{axis}
			\end{tikzpicture}
			\caption{CG/ASM solve time}
			\label{2dasmsolve}
		\end{subfigure}
		\begin{subfigure}[c]{0.3\textwidth}
			\begin{tikzpicture}[scale=0.45]
				\begin{axis}[xlabel={Number of vertices},
					ylabel=Time(s),
					xmode=log, ymode=log,
					ymin=7e-2, ymax=1e2,
					log basis x={10},
					legend pos = north west,
					legend style = {legend columns = 2}]
					\addplot[dotted,mark=o, mark options={solid}]
					table [x=NV, y expr=\thisrow{KSPSOLVE}+\thisrow{PCSETUP}, col sep=comma]{code/results_2d_s2_p1pc.csv};
					\addlegendentry{$\mathcal{S}_2$};
					\addplot[dashed,mark=o, mark options={solid}]
					table [x=NV, y expr=\thisrow{KSPSOLVE}+\thisrow{PCSETUP}, col sep=comma]{code/results_2d_q2_p1pc.csv};
					\addlegendentry{$\mathcal{Q}_2$};
					\addplot[dotted,mark=triangle, mark options={solid}]
					table [x=NV, y expr=\thisrow{KSPSOLVE}+\thisrow{PCSETUP}, col sep=comma]{code/results_2d_s3_p1pc.csv};
					\addlegendentry{$\mathcal{S}_3$};
					\addplot[dashed,mark=triangle, mark options={solid}]
					table [x=NV, y expr=\thisrow{KSPSOLVE}+\thisrow{PCSETUP}, col sep=comma]{code/results_2d_q3_p1pc.csv};
					\addlegendentry{$\mathcal{Q}_3$};
					\addplot[dotted,mark=square, mark options={solid}]
					table [x=NV, y expr=\thisrow{KSPSOLVE}+\thisrow{PCSETUP}, col sep=comma]{code/results_2d_s4_p1pc.csv};
					\addlegendentry{$\mathcal{S}_4$};
					\addplot[dashed,mark=square, mark options={solid}]
					table [x=NV, y expr=\thisrow{KSPSOLVE}+\thisrow{PCSETUP}, col sep=comma]{code/results_2d_q4_p1pc.csv};
					\addlegendentry{$\mathcal{Q}_4$};
				\end{axis}
			\end{tikzpicture}
			\caption{Setup+solve}
			\label{2dasmsetupsolve}
		\end{subfigure} \\
		\begin{subfigure}[l]{0.31\textwidth}
			\begin{tikzpicture}[scale=0.45]
				\begin{axis}[xlabel={Number of vertices},
					ylabel=Time(s),
					xmode=log, ymode=log, ymin=7e-2, ymax=1e2,
					log basis x={10},
					legend pos = north west,
					legend style = {legend columns = 2}]
					\addplot[dotted,mark=o, mark options={solid}]
					table [x=NV, y=PCSETUP, col sep=comma]{code/results_2d_s2_mg.csv};
					\addlegendentry{$\mathcal{S}_2$};
					\addplot[dashed,mark=o, mark options={solid}]
					table [x=NV, y=PCSETUP, col sep=comma]{code/results_2d_q2_mg.csv};
					\addlegendentry{$\mathcal{Q}_2$};
					\addplot[dotted,mark=triangle, mark options={solid}]
					table [x=NV, y=PCSETUP, col sep=comma]{code/results_2d_s3_mg.csv};
					\addlegendentry{$\mathcal{S}_3$};
					\addplot[dashed,mark=triangle, mark options={solid}]
					table [x=NV, y=PCSETUP, col sep=comma]{code/results_2d_q3_mg.csv};
					\addlegendentry{$\mathcal{Q}_3$};
					\addplot[dotted,mark=square, mark options={solid}]
					table [x=NV, y=PCSETUP, col sep=comma]{code/results_2d_s4_mg.csv};
					\addlegendentry{$\mathcal{S}_4$};
					\addplot[dashed,mark=square, mark options={solid}]
					table [x=NV, y=PCSETUP, col sep=comma]{code/results_2d_q4_mg.csv};
					\addlegendentry{$\mathcal{Q}_4$};
				\end{axis}
			\end{tikzpicture}
			\caption{MG Setup time}
			\label{2dmgsetup}    
		\end{subfigure}
		\begin{subfigure}[c]{0.3\textwidth}
			\begin{tikzpicture}[scale=0.45]
				\begin{axis}[xlabel={Number of vertices},
					ylabel=Time(s),
					xmode=log, ymode=log, ymin=7e-2, ymax=1e2,
					log basis x={10},
					legend pos = north west,
					legend style = {legend columns = 2}]
					\addplot[dotted,mark=o, mark options={solid}]
					table [x=NV, y=KSPSOLVE, col sep=comma]{code/results_2d_s2_mg.csv};
					\addlegendentry{$\mathcal{S}_2$};
					\addplot[dashed,mark=o, mark options={solid}]
					table [x=NV, y=KSPSOLVE, col sep=comma]{code/results_2d_q2_mg.csv};
					\addlegendentry{$\mathcal{Q}_2$};
					\addplot[dotted,mark=triangle, mark options={solid}]
					table [x=NV, y=KSPSOLVE, col sep=comma]{code/results_2d_s3_mg.csv};
					\addlegendentry{$\mathcal{S}_3$};
					\addplot[dashed,mark=triangle, mark options={solid}]
					table [x=NV, y=KSPSOLVE, col sep=comma]{code/results_2d_q3_mg.csv};
					\addlegendentry{$\mathcal{Q}_3$};
					\addplot[dotted,mark=square, mark options={solid}]
					table [x=NV, y=KSPSOLVE, col sep=comma]{code/results_2d_s4_mg.csv};
					\addlegendentry{$\mathcal{S}_4$};
					\addplot[dashed,mark=square, mark options={solid}]
					table [x=NV, y=KSPSOLVE, col sep=comma]{code/results_2d_q4_mg.csv};
					\addlegendentry{$\mathcal{Q}_4$};
				\end{axis}
			\end{tikzpicture}
			\caption{CG/MG solve time}
			\label{2dmgsolve}
		\end{subfigure}
		\begin{subfigure}[c]{0.3\textwidth}
			\begin{tikzpicture}[scale=0.45]
				\begin{axis}[xlabel={Number of vertices},
					ylabel=Time(s),
					xmode=log, ymode=log, ymin=7e-2, ymax=1e2,
					log basis x={10},
					legend pos = north west,
					legend style = {legend columns = 2}]
					\addplot[dotted,mark=o, mark options={solid}]
					table [x=NV, y expr=\thisrow{KSPSOLVE}+\thisrow{PCSETUP}, col sep=comma]{code/results_2d_s2_mg.csv};
					\addlegendentry{$\mathcal{S}_2$};
					\addplot[dashed,mark=o, mark options={solid}]
					table [x=NV, y expr=\thisrow{KSPSOLVE}+\thisrow{PCSETUP}, col sep=comma]{code/results_2d_q2_mg.csv};
					\addlegendentry{$\mathcal{Q}_2$};
					\addplot[dotted,mark=triangle, mark options={solid}]
					table [x=NV, y expr=\thisrow{KSPSOLVE}+\thisrow{PCSETUP}, col sep=comma]{code/results_2d_s3_mg.csv};
					\addlegendentry{$\mathcal{S}_3$};
					\addplot[dashed,mark=triangle, mark options={solid}]
					table [x=NV, y expr=\thisrow{KSPSOLVE}+\thisrow{PCSETUP}, col sep=comma]{code/results_2d_q3_mg.csv};
					\addlegendentry{$\mathcal{Q}_3$};
					\addplot[dotted,mark=square, mark options={solid}]
					table [x=NV, y expr=\thisrow{KSPSOLVE}+\thisrow{PCSETUP}, col sep=comma]{code/results_2d_s4_mg.csv};
					\addlegendentry{$\mathcal{S}_4$};
					\addplot[dashed,mark=square, mark options={solid}]
					table [x=NV, y expr=\thisrow{KSPSOLVE}+\thisrow{PCSETUP}, col sep=comma]{code/results_2d_q4_mg.csv};
					\addlegendentry{$\mathcal{Q}_4$};
				\end{axis}
			\end{tikzpicture}
			\caption{Setup+solve}
			\label{2dmgsetupsolve}    
		\end{subfigure}
		\caption{Solver time for 2D Poisson on unstructured mesh.  The first plots give the time to set up the ASM or MG preconditioner in PETSc.  The middle column gives the time to solve the system using preconditioned conjugate gradients.  The final column gives the sum of the first two.}
		\label{fig:2dpoissontime}
	\end{figure}
	
	Since Firedrake does not yet support proper hexahedral meshes, we test our solvers in three dimensions for the Poisson problem on the unit cube, choosing Dirichlet boundary conditions and forcing function to match a known smooth solution.
	As in 2D, we approximate the solution with $\mathcal{Q}_k$ and $\mathcal{S}_k$ of degrees 2, 3, and 4.  We begin with an initial $4 \times 4 \times 4$ coarse mesh and uniformly refine it, reporting simulations on an $N \times N \times N$ mesh with $N=8, 16, 32, 64$.  Again, we partition the mesh over 16 cores of the workstation.  We remark that Firedrake distributes extruded meshes according in a two-dimensional fashion, so that all cells in a vertical column appear on the same processor.  This does not appear to be a major limitation in our particular situation.

	The solver configuration for three dimensions is nearly identical to that for two dimensions, other than that we modified the \texttt{ASMStarPC} class to properly identify patches on extruded meshes.
	\cref{Poisson3dits} shows only mild variation of the conjugate gradient iteration count for the various discretizations.
	As with the two-dimensional case, \cref{3dasmsetup} and~\cref{3dmgsetup} also show a considerable overhead in the setup phase for the preconditioners.
	Moreover, the overall cost to set up the ASM preconditioner seems much higher than that for multigrid.
	However, \cref{3dasmsolve} and~\cref{3dmgsolve} show that the reverse is true for solving the linear system.
	\cref{3dasmsetupsolve} and~\cref{3dmgsetupsolve} show that the total cost for setup plus a single solve is somewhat cheaper for the ASM-preconditioned method, so we recommend the additive Schwarz method, especially when the system must be solved repeatedly.
	In \cref{3dasmsolve}, we do see a small win for $\mathcal{S}_2$ elements over $\mathcal{Q}_2$.
        We also observe that the cost of serendipity elements tends to track a lower-order $\mathcal{Q}_k$ element, despite giving a higher order of accuracy -- $\mathcal{S}_3$ costs about the same $\mathcal{Q}_2$ and $\mathcal{S}_4$ about the same as $\mathcal{Q}_3$.  These relative run times seem to be in accordance with \cref{fig:3dpoissondofs}.	
	
	\begin{figure}
		\begin{subfigure}[l]{0.48\textwidth}
			\begin{tikzpicture}[scale=0.65]
				\begin{axis}[xlabel={Number of vertices},
					ylabel=Iterations,
					xmode=log,
					log basis x={10}, ymin=1, ymax=15,
					legend pos = south west,
					legend style = {legend columns = 2}]
					\addplot[dotted,mark=o, mark options={solid}]
					table [x=NV, y=ITS, col sep=comma]{code/results_3d_s2_p1pc.csv};
					\addlegendentry{$\mathcal{S}_2$};
					\addplot[dashed,mark=o, mark options={solid}]
					table [x=NV, y=ITS, col sep=comma]{code/results_3d_q2_p1pc.csv};
					\addlegendentry{$\mathcal{Q}_2$};
					\addplot[dotted,mark=triangle, mark options={solid}]
					table [x=NV, y=ITS, col sep=comma]{code/results_3d_s3_p1pc.csv};
					\addlegendentry{$\mathcal{S}_3$};
					\addplot[dashed,mark=triangle, mark options={solid}]
					table [x=NV, y=ITS, col sep=comma]{code/results_3d_q3_p1pc.csv};
					\addlegendentry{$\mathcal{Q}_3$};
					\addplot[dotted,mark=square, mark options={solid}]
					table [x=NV, y=ITS, col sep=comma]{code/results_3d_s4_p1pc.csv};
					\addlegendentry{$\mathcal{S}_4$};
					\addplot[dashed,mark=square, mark options={solid}]
					table [x=NV, y=ITS, col sep=comma]{code/results_3d_q4_p1pc.csv};
					\addlegendentry{$\mathcal{Q}_4$};				
				\end{axis}
			\end{tikzpicture}
			\label{ASMPoisson3dits}
			\caption{ASM}
		\end{subfigure}
		\begin{subfigure}[l]{0.48\textwidth}
			\begin{tikzpicture}[scale=0.65]
				\begin{axis}[xlabel={Number of vertices},
					ylabel=Iterations,
					xmode=log,
					log basis x={10}, ymin=1, ymax=15,
					legend pos = south west,
					legend style = {legend columns = 2}]
					\addplot[dotted,mark=o, mark options={solid}]
					table [x=NV, y=ITS, col sep=comma]{code/results_3d_s2_mg.csv};
					\addlegendentry{$\mathcal{S}_2$};
					\addplot[dashed,mark=o, mark options={solid}]
					table [x=NV, y=ITS, col sep=comma]{code/results_3d_q2_mg.csv};
					\addlegendentry{$\mathcal{Q}_2$};
					\addplot[dotted,mark=triangle, mark options={solid}]
					table [x=NV, y=ITS, col sep=comma]{code/results_3d_s3_mg.csv};
					\addlegendentry{$\mathcal{S}_3$};
					\addplot[dashed,mark=triangle, mark options={solid}]
					table [x=NV, y=ITS, col sep=comma]{code/results_3d_q3_mg.csv};
					\addlegendentry{$\mathcal{Q}_3$};
					\addplot[dotted,mark=square, mark options={solid}]
					table [x=NV, y=ITS, col sep=comma]{code/results_3d_s4_mg.csv};
					\addlegendentry{$\mathcal{S}_4$};
					\addplot[dashed,mark=square, mark options={solid}]
					table [x=NV, y=ITS, col sep=comma]{code/results_3d_q4_mg.csv};
					\addlegendentry{$\mathcal{Q}_4$};
				\end{axis}
			\end{tikzpicture}
			\label{MGPoisson3dits}
			\caption{MG}
		\end{subfigure}	
		\caption{Iteration counts for ASM and MG-preconditioned conjugate gradients for the 3D Poisson problem.}
		\label{Poisson3dits}
	\end{figure}
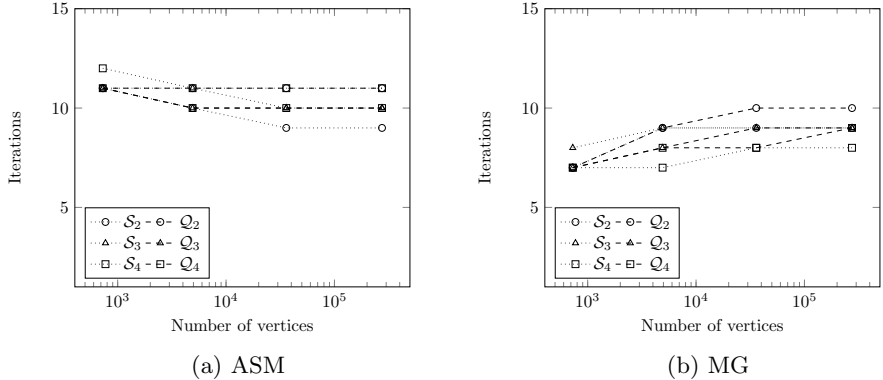

	\begin{figure}
		\begin{subfigure}[l]{0.31\textwidth}
			\begin{tikzpicture}[scale=0.45]
				\begin{axis}[xlabel={Number of vertices},
					ylabel=Time(s),
					xmode=log, ymode=log,
					log basis x={10},
					ymin=8e-1, ymax=1e4,
					legend pos = north west,
					legend style = {legend columns = 2}]
					\addplot[dotted,mark=o, mark options={solid}]
					table [x=NV, y=PCSETUP, col sep=comma]{code/results_3d_s2_p1pc.csv};
					\addlegendentry{$\mathcal{S}_2$};
					\addplot[dashed,mark=o, mark options={solid}]
					table [x=NV, y=PCSETUP, col sep=comma]{code/results_3d_q2_p1pc.csv};
					\addlegendentry{$\mathcal{Q}_2$};
					\addplot[dotted,mark=triangle, mark options={solid}]
					table [x=NV, y=PCSETUP, col sep=comma]{code/results_3d_s3_p1pc.csv};
					\addlegendentry{$\mathcal{S}_3$};
					\addplot[dashed,mark=triangle, mark options={solid}]
					table [x=NV, y=PCSETUP, col sep=comma]{code/results_3d_q3_p1pc.csv};
					\addlegendentry{$\mathcal{Q}_3$};
					\addplot[dotted,mark=square, mark options={solid}]
					table [x=NV, y=PCSETUP, col sep=comma]{code/results_3d_s4_p1pc.csv};
					\addlegendentry{$\mathcal{S}_4$};
					\addplot[dashed,mark=square, mark options={solid}]
					table [x=NV, y=PCSETUP, col sep=comma]{code/results_3d_q4_p1pc.csv};
					\addlegendentry{$\mathcal{Q}_4$};
				\end{axis}
			\end{tikzpicture}
			\caption{ASM Setup time}
			\label{3dasmsetup}
		\end{subfigure}
		\begin{subfigure}[c]{0.3\textwidth}
			\begin{tikzpicture}[scale=0.45]
				\begin{axis}[xlabel={Number of vertices},
					ylabel=Time(s),
					xmode=log, ymode=log,
					ymin=8e-1, ymax=1e4,
					log basis x={10},
					legend pos = north west,
					legend style = {legend columns = 2}]
					\addplot[dotted,mark=o, mark options={solid}]
					table [x=NV, y=KSPSOLVE, col sep=comma]{code/results_3d_s2_p1pc.csv};
					\addlegendentry{$\mathcal{S}_2$};
					\addplot[dashed,mark=o, mark options={solid}]
					table [x=NV, y=KSPSOLVE, col sep=comma]{code/results_3d_q2_p1pc.csv};
					\addlegendentry{$\mathcal{Q}_2$};
					\addplot[dotted,mark=triangle, mark options={solid}]
					table [x=NV, y=KSPSOLVE, col sep=comma]{code/results_3d_s3_p1pc.csv};
					\addlegendentry{$\mathcal{S}_3$};
					\addplot[dashed,mark=triangle, mark options={solid}]
					table [x=NV, y=KSPSOLVE, col sep=comma]{code/results_3d_q3_p1pc.csv};
					\addlegendentry{$\mathcal{Q}_3$};
					\addplot[dotted,mark=square, mark options={solid}]
					table [x=NV, y =KSPSOLVE, col sep=comma]{code/results_3d_s4_p1pc.csv};
					\addlegendentry{$\mathcal{S}_4$};
					\addplot[dashed,mark=square, mark options={solid}]
					table [x=NV, y=KSPSOLVE, col sep=comma]{code/results_3d_q4_p1pc.csv};
					\addlegendentry{$\mathcal{Q}_4$};
				\end{axis}
			\end{tikzpicture}
			\caption{CG/ASM solve time}
			\label{3dasmsolve}
		\end{subfigure}
		\begin{subfigure}[c]{0.3\textwidth}
			\begin{tikzpicture}[scale=0.45]
				\begin{axis}[xlabel={Number of vertices},
					ylabel=Time(s),
					xmode=log, ymode=log,
					ymin=8e-1, ymax=1e4,
					log basis x={10},
					legend pos = north west,
					legend style = {legend columns = 2}]
					\addplot[dotted,mark=o, mark options={solid}]
					table [x=NV, y expr=\thisrow{KSPSOLVE}+\thisrow{PCSETUP}, col sep=comma]{code/results_3d_s2_p1pc.csv};
					\addlegendentry{$\mathcal{S}_2$};
					\addplot[dashed,mark=o, mark options={solid}]
					table [x=NV, y expr=\thisrow{KSPSOLVE}+\thisrow{PCSETUP}, col sep=comma]{code/results_3d_q2_p1pc.csv};
					\addlegendentry{$\mathcal{Q}_2$};
					\addplot[dotted,mark=triangle, mark options={solid}]
					table [x=NV, y expr=\thisrow{KSPSOLVE}+\thisrow{PCSETUP}, col sep=comma]{code/results_3d_s3_p1pc.csv};
					\addlegendentry{$\mathcal{S}_3$};
					\addplot[dashed,mark=triangle, mark options={solid}]
					table [x=NV, y expr=\thisrow{KSPSOLVE}+\thisrow{PCSETUP}, col sep=comma]{code/results_3d_q3_p1pc.csv};
					\addlegendentry{$\mathcal{Q}_3$};
					\addplot[dotted,mark=square, mark options={solid}]
					table [x=NV, y expr=\thisrow{KSPSOLVE}+\thisrow{PCSETUP}, col sep=comma]{code/results_3d_s4_p1pc.csv};
					\addlegendentry{$\mathcal{S}_4$};
					\addplot[dashed,mark=square, mark options={solid}]
					table [x=NV, y expr=\thisrow{KSPSOLVE}+\thisrow{PCSETUP}, col sep=comma]{code/results_3d_q4_p1pc.csv};
					\addlegendentry{$\mathcal{Q}_4$};
				\end{axis}
			\end{tikzpicture}
			\caption{Setup+solve}
			\label{3dasmsetupsolve}
		\end{subfigure} \\
		\begin{subfigure}[l]{0.31\textwidth}
			\begin{tikzpicture}[scale=0.45]
				\begin{axis}[xlabel={Number of vertices},
					ylabel=Time(s),
					xmode=log, ymode=log,
					ymin=8e-1, ymax=1e4,          
					log basis x={10},
					legend pos = north west,
					legend style = {legend columns = 2}]
					\addplot[dotted,mark=o, mark options={solid}]
					table [x=NV, y=PCSETUP, col sep=comma]{code/results_3d_s2_mg.csv};
					\addlegendentry{$\mathcal{S}_2$};
					\addplot[dashed,mark=o, mark options={solid}]
					table [x=NV, y=PCSETUP, col sep=comma]{code/results_3d_q2_mg.csv};
					\addlegendentry{$\mathcal{Q}_2$};
					\addplot[dotted,mark=triangle, mark options={solid}]
					table [x=NV, y=PCSETUP, col sep=comma]{code/results_3d_s3_mg.csv};
					\addlegendentry{$\mathcal{S}_3$};
					\addplot[dashed,mark=triangle, mark options={solid}]
					table [x=NV, y=PCSETUP, col sep=comma]{code/results_3d_q3_mg.csv};
					\addlegendentry{$\mathcal{Q}_3$};
					\addplot[dotted,mark=square, mark options={solid}]
					table [x=NV, y=PCSETUP, col sep=comma]{code/results_3d_s4_mg.csv};
					\addlegendentry{$\mathcal{S}_4$};
					\addplot[dashed,mark=square, mark options={solid}]
					table [x=NV, y=PCSETUP, col sep=comma]{code/results_3d_q4_mg.csv};
					\addlegendentry{$\mathcal{Q}_4$};
				\end{axis}
			\end{tikzpicture}
			\caption{MG Setup time}
			\label{3dmgsetup} 
		\end{subfigure}
		\begin{subfigure}[c]{0.3\textwidth}
			\begin{tikzpicture}[scale=0.45]
				\begin{axis}[xlabel={Number of vertices},
					ylabel=Time(s),
					xmode=log, ymode=log,
					ymin=8e-1, ymax=1e4,
					log basis x={10},
					legend pos = north west,
					legend style = {legend columns = 2}]
					\addplot[dotted,mark=o, mark options={solid}]
					table [x=NV, y=KSPSOLVE, col sep=comma]{code/results_3d_s2_mg.csv};
					\addlegendentry{$\mathcal{S}_2$};
					\addplot[dashed,mark=o, mark options={solid}]
					table [x=NV, y=KSPSOLVE, col sep=comma]{code/results_3d_q2_mg.csv};
					\addlegendentry{$\mathcal{Q}_2$};
					\addplot[dotted,mark=triangle, mark options={solid}]
					table [x=NV, y=KSPSOLVE, col sep=comma]{code/results_3d_s3_mg.csv};
					\addlegendentry{$\mathcal{S}_3$};
					\addplot[dashed,mark=triangle, mark options={solid}]
					table [x=NV, y=KSPSOLVE, col sep=comma]{code/results_3d_q3_mg.csv};
					\addlegendentry{$\mathcal{Q}_3$};
					\addplot[dotted,mark=square, mark options={solid}]
					table [x=NV, y=KSPSOLVE, col sep=comma]{code/results_3d_s4_mg.csv};
					\addlegendentry{$\mathcal{S}_4$};
					\addplot[dashed,mark=square, mark options={solid}]
					table [x=NV, y=KSPSOLVE, col sep=comma]{code/results_3d_q4_mg.csv};
					\addlegendentry{$\mathcal{Q}_4$};
				\end{axis}
			\end{tikzpicture}
			\caption{CG/MG solve time}
			\label{3dmgsolve}   
		\end{subfigure}
		\begin{subfigure}[c]{0.3\textwidth}
			\begin{tikzpicture}[scale=0.45]
				\begin{axis}[xlabel={Number of vertices},
					ylabel=Time(s),
					xmode=log, ymode=log,
					ymin=8e-1, ymax=1e4,
					log basis x={10},
					legend pos = north west,
					legend style = {legend columns = 2}]
					\addplot[dotted,mark=o, mark options={solid}]
					table [x=NV, y expr=\thisrow{KSPSOLVE}+\thisrow{PCSETUP}, col sep=comma]{code/results_3d_s2_mg.csv};
					\addlegendentry{$\mathcal{S}_2$};
					\addplot[dashed,mark=o, mark options={solid}]
					table [x=NV, y expr=\thisrow{KSPSOLVE}+\thisrow{PCSETUP}, col sep=comma]{code/results_3d_q2_mg.csv};
					\addlegendentry{$\mathcal{Q}_2$};
					\addplot[dotted,mark=triangle, mark options={solid}]
					table [x=NV, y expr=\thisrow{KSPSOLVE}+\thisrow{PCSETUP}, col sep=comma]{code/results_3d_s3_mg.csv};
					\addlegendentry{$\mathcal{S}_3$};
					\addplot[dashed,mark=triangle, mark options={solid}]
					table [x=NV, y expr=\thisrow{KSPSOLVE}+\thisrow{PCSETUP}, col sep=comma]{code/results_3d_q3_mg.csv};
					\addlegendentry{$\mathcal{Q}_3$};
					\addplot[dotted,mark=square, mark options={solid}]
					table [x=NV, y expr=\thisrow{KSPSOLVE}+\thisrow{PCSETUP}, col sep=comma]{code/results_3d_s4_mg.csv};
					\addlegendentry{$\mathcal{S}_4$};
					\addplot[dashed,mark=square, mark options={solid}]
					table [x=NV, y expr=\thisrow{KSPSOLVE}+\thisrow{PCSETUP}, col sep=comma]{code/results_3d_q4_mg.csv};
					\addlegendentry{$\mathcal{Q}_4$};
				\end{axis}
			\end{tikzpicture}
			\caption{Setup+solve}
			\label{3dmgsetupsolve}   
		\end{subfigure}
		\label{fig:3dpoissontime}
		\caption{Solver time for 3D Poisson.  The first plots give the time to set up the ASM or MG preconditioner in PETSc.  The middle column gives the solver time for preconditioned CG, and the final column sums the first two.}
	\end{figure}
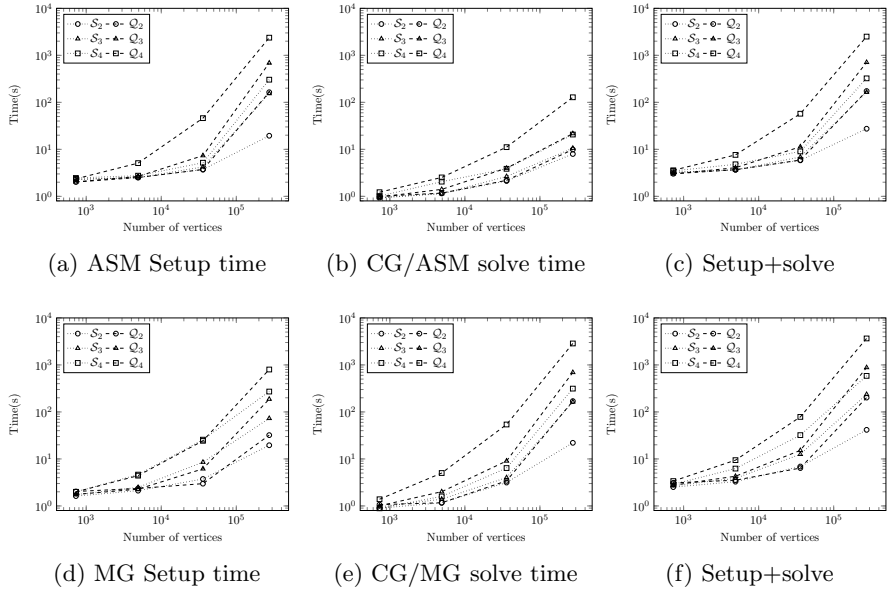

	We also test our preconditioners on the planar elasticty problem~\eqref{elasticity}.
	We consider a cantilever beam with $\Omega = [0,25] \times [0, 1]$, posing homogeneous Dirichlet (no-displacement) boundary conditions on the left end ($x=0$) and no-stress boundary conditions on the rest of the boundary.
	The forcing function is taken to be a constant downward force of the form $f(x, y) = (0, -g)$.
        We divided $\Omega$ into a $125 \times 5$ mesh of squares and performed a sequence of uniform refinements to estabish a multigrid hierarchy.  All runs are partitioned over 24 cores of the workstation.
	Although the problem is only run in 2D, the vector-valued nature of the problem makes larger patch problems and hence accentuates the gains obtained by reducing the basis size.  Hence one may hope to see a nice gain from serendipity elements.
	
	The iteration counts for both additive Schwarz and multigrid methods shown in \cref{fig:2delasticityits} are also in line with the theory, quite comparable to the results above for the Laplace operator.
	Timing results are showin in \cref{fig:elasticitytime}.  Again, we see some overhead in setting up the preconditioners, especially the ASM ones.  However, once the ASM preconditioners are set up, there is a clear win in solution time for serendipity elements, with $\mathcal{S}_k$ outperforming $\mathcal{Q}_k$ for each $k$.
	Moreover, $\mathcal{S}_4$ turns out to be cheaper than $\mathcal{Q}_3$.
        Although the multigrid algorithm for serendipity elements is not yet as performant as we would like (again, inter-grid transfers are much slower than for $\mathcal{Q}_k$ elements), on finer meshes additive Schwarz does give some clear advantages for serendipity, and we see that $\mathcal{S}_3$ is only slightly more expensive than $\mathcal{Q}_2$ and that $\mathcal{S}_4$ is somewhat cheaper than $\mathcal{Q}_3$ in total run time.	
	
	\begin{figure}
		\begin{subfigure}[l]{0.48\textwidth}
			\begin{tikzpicture}[scale=0.65]
				\begin{axis}[xlabel={Number of vertices},
					ylabel=Iterations,
					xmode=log,
					log basis x={10}, ymin=1, ymax=15,
					legend pos = south west,
					legend style = {legend columns = 2}]
					\addplot[dotted,mark=o, mark options={solid}]
					table [x=NV, y=ITS, col sep=comma]{code/results_elasticity_s2_p1pc.csv};
					\addlegendentry{$\mathcal{S}_2$};
					\addplot[dashed,mark=o, mark options={solid}]
					table [x=NV, y=ITS, col sep=comma]{code/results_elasticity_q2_p1pc.csv};
					\addlegendentry{$\mathcal{Q}_2$};
					\addplot[dotted,mark=triangle, mark options={solid}]
					table [x=NV, y=ITS, col sep=comma]{code/results_elasticity_s3_p1pc.csv};
					\addlegendentry{$\mathcal{S}_3$};
					\addplot[dashed,mark=triangle, mark options={solid}]
					table [x=NV, y=ITS, col sep=comma]{code/results_elasticity_q3_p1pc.csv};
					\addlegendentry{$\mathcal{Q}_3$};
					\addplot[dotted,mark=square, mark options={solid}]
					table [x=NV, y=ITS, col sep=comma]{code/results_elasticity_s4_p1pc.csv};
					\addlegendentry{$\mathcal{S}_4$};
					\addplot[dashed,mark=square, mark options={solid}]
					table [x=NV, y=ITS, col sep=comma]{code/results_elasticity_q4_p1pc.csv};
					\addlegendentry{$\mathcal{Q}_4$};				
				\end{axis}
			\end{tikzpicture}
			\label{ASMElasticity2dits}
			\caption{ASM}
		\end{subfigure}
		\begin{subfigure}[l]{0.48\textwidth}
			\begin{tikzpicture}[scale=0.65]
				\begin{axis}[xlabel={Number of vertices},
					ylabel=Iterations,
					xmode=log,
					log basis x={10}, ymin=1, ymax=15,
					legend pos = south west,
					legend style = {legend columns = 2}]
					\addplot[dotted,mark=o, mark options={solid}]
					table [x=NV, y=ITS, col sep=comma]{code/results_elasticity_s2_mg.csv};
					\addlegendentry{$\mathcal{S}_2$};
					\addplot[dashed,mark=o, mark options={solid}]
					table [x=NV, y=ITS, col sep=comma]{code/results_elasticity_q2_mg.csv};
					\addlegendentry{$\mathcal{Q}_2$};
					\addplot[dotted,mark=triangle, mark options={solid}]
					table [x=NV, y=ITS, col sep=comma]{code/results_elasticity_s3_mg.csv};
					\addlegendentry{$\mathcal{S}_3$};
					\addplot[dashed,mark=triangle, mark options={solid}]
					table [x=NV, y=ITS, col sep=comma]{code/results_elasticity_q3_mg.csv};
					\addlegendentry{$\mathcal{Q}_3$};
					\addplot[dotted,mark=square, mark options={solid}]
					table [x=NV, y=ITS, col sep=comma]{code/results_elasticity_s4_mg.csv};
					\addlegendentry{$\mathcal{S}_4$};
					\addplot[dashed,mark=square, mark options={solid}]
					table [x=NV, y=ITS, col sep=comma]{code/results_elasticity_q4_mg.csv};
					\addlegendentry{$\mathcal{Q}_4$};				
				\end{axis}
			\end{tikzpicture}
			\label{MGElasticity2dits}
			\caption{MG}
		\end{subfigure}	
		\caption{Iteration counts for ASM and MG-preconditioned conjugate gradients for the planar elasticity problem.}
		\label{fig:2delasticityits}
	\end{figure}
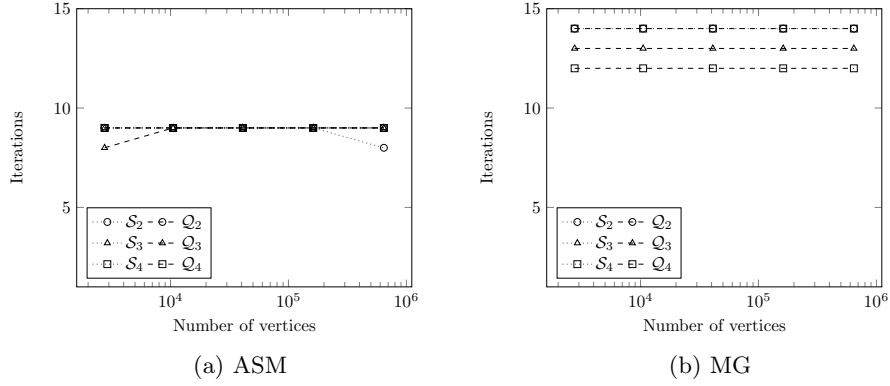

	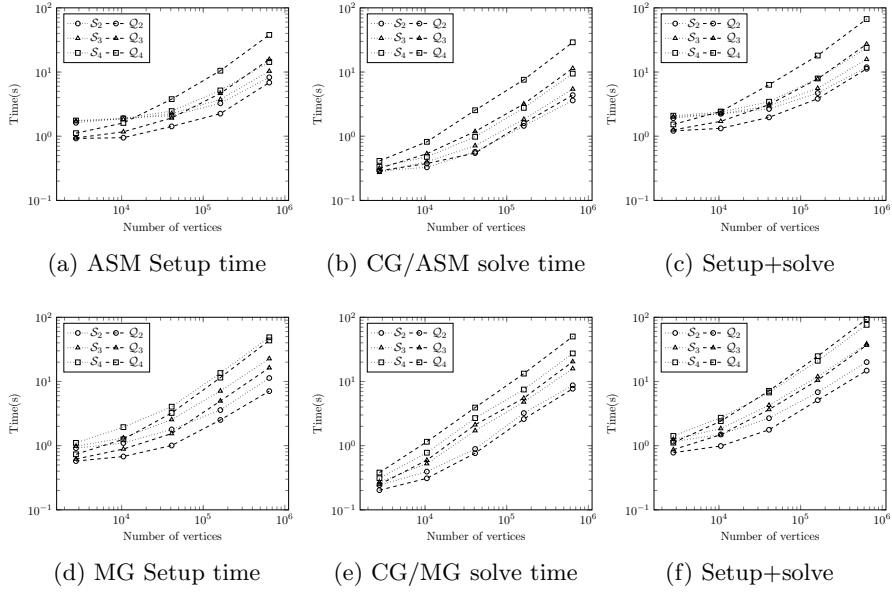
\begin{figure}
		\begin{subfigure}[l]{0.31\textwidth}
			\begin{tikzpicture}[scale=0.45]
				\begin{axis}[xlabel={Number of vertices},
					ylabel=Time(s),
					xmode=log, ymode=log,
					log basis x={10},
					ymin=1e-1, ymax=1e2,
					legend pos = north west,
					legend style = {legend columns = 2}]
					\addplot[dotted,mark=o, mark options={solid}]
					table [x=NV, y=PCSETUP, col sep=comma]{code/results_elasticity_s2_p1pc.csv};
					\addlegendentry{$\mathcal{S}_2$};
					\addplot[dashed,mark=o, mark options={solid}]
					table [x=NV, y=PCSETUP, col sep=comma]{code/results_elasticity_q2_p1pc.csv};
					\addlegendentry{$\mathcal{Q}_2$};
					\addplot[dotted,mark=triangle, mark options={solid}]
					table [x=NV, y=PCSETUP, col sep=comma]{code/results_elasticity_s3_p1pc.csv};
					\addlegendentry{$\mathcal{S}_3$};
					\addplot[dashed,mark=triangle, mark options={solid}]
					table [x=NV, y=PCSETUP, col sep=comma]{code/results_elasticity_q3_p1pc.csv};
					\addlegendentry{$\mathcal{Q}_3$};
					\addplot[dotted,mark=square, mark options={solid}]
					table [x=NV, y=PCSETUP, col sep=comma]{code/results_elasticity_s4_p1pc.csv};
					\addlegendentry{$\mathcal{S}_4$};
					\addplot[dashed,mark=square, mark options={solid}]
					table [x=NV, y=PCSETUP, col sep=comma]{code/results_elasticity_q4_p1pc.csv};
					\addlegendentry{$\mathcal{Q}_4$};
				\end{axis}
			\end{tikzpicture}
			\caption{ASM Setup time}
			\label{elasticityasmsetup}
		\end{subfigure}
		\begin{subfigure}[c]{0.3\textwidth}
			\begin{tikzpicture}[scale=0.45]
				\begin{axis}[xlabel={Number of vertices},
					ylabel=Time(s),
					xmode=log, ymode=log,
					log basis x={10},
					ymin=1e-1, ymax=1e2,          
					legend pos = north west,
					legend style = {legend columns = 2}]
					\addplot[dotted,mark=o, mark options={solid}]
					table [x=NV, y=KSPSOLVE, col sep=comma]{code/results_elasticity_s2_p1pc.csv};
					\addlegendentry{$\mathcal{S}_2$};
					\addplot[dashed,mark=o, mark options={solid}]
					table [x=NV, y=KSPSOLVE, col sep=comma]{code/results_elasticity_q2_p1pc.csv};
					\addlegendentry{$\mathcal{Q}_2$};
					\addplot[dotted,mark=triangle, mark options={solid}]
					table [x=NV, y=KSPSOLVE, col sep=comma]{code/results_elasticity_s3_p1pc.csv};
					\addlegendentry{$\mathcal{S}_3$};
					\addplot[dashed,mark=triangle, mark options={solid}]
					table [x=NV, y=KSPSOLVE, col sep=comma]{code/results_elasticity_q3_p1pc.csv};
					\addlegendentry{$\mathcal{Q}_3$};
					\addplot[dotted,mark=square, mark options={solid}]
					table [x=NV, y =KSPSOLVE, col sep=comma]{code/results_elasticity_s4_p1pc.csv};
					\addlegendentry{$\mathcal{S}_4$};
					\addplot[dashed,mark=square, mark options={solid}]
					table [x=NV, y=KSPSOLVE, col sep=comma]{code/results_elasticity_q4_p1pc.csv};
					\addlegendentry{$\mathcal{Q}_4$};
				\end{axis}
			\end{tikzpicture}
			\caption{CG/ASM solve time}
			\label{elasticityasmsolve}
		\end{subfigure}
		\begin{subfigure}[c]{0.3\textwidth}
			\begin{tikzpicture}[scale=0.45]
				\begin{axis}[xlabel={Number of vertices},
					ylabel=Time(s),
					xmode=log, ymode=log,
					ymin=1e-1, ymax=1e2,
					log basis x={10},
					legend pos = north west,
					legend style = {legend columns = 2}]
					\addplot[dotted,mark=o, mark options={solid}]
					table [x=NV, y expr=\thisrow{KSPSOLVE}+\thisrow{PCSETUP}, col sep=comma]{code/results_elasticity_s2_p1pc.csv};
					\addlegendentry{$\mathcal{S}_2$};
					\addplot[dashed,mark=o, mark options={solid}]
					table [x=NV, y expr=\thisrow{KSPSOLVE}+\thisrow{PCSETUP}, col sep=comma]{code/results_elasticity_q2_p1pc.csv};
					\addlegendentry{$\mathcal{Q}_2$};
					\addplot[dotted,mark=triangle, mark options={solid}]
					table [x=NV, y expr=\thisrow{KSPSOLVE}+\thisrow{PCSETUP}, col sep=comma]{code/results_elasticity_s3_p1pc.csv};
					\addlegendentry{$\mathcal{S}_3$};
					\addplot[dashed,mark=triangle, mark options={solid}]
					table [x=NV, y expr=\thisrow{KSPSOLVE}+\thisrow{PCSETUP}, col sep=comma]{code/results_elasticity_q3_p1pc.csv};
					\addlegendentry{$\mathcal{Q}_3$};
					\addplot[dotted,mark=square, mark options={solid}]
					table [x=NV, y expr=\thisrow{KSPSOLVE}+\thisrow{PCSETUP}, col sep=comma]{code/results_elasticity_s4_p1pc.csv};
					\addlegendentry{$\mathcal{S}_4$};
					\addplot[dashed,mark=square, mark options={solid}]
					table [x=NV, y expr=\thisrow{KSPSOLVE}+\thisrow{PCSETUP}, col sep=comma]{code/results_elasticity_q4_p1pc.csv};
					\addlegendentry{$\mathcal{Q}_4$};
				\end{axis}
			\end{tikzpicture}
			\caption{Setup+solve}
			\label{elasticityasmsetupsolve}
		\end{subfigure} \\
		\begin{subfigure}[l]{0.31\textwidth}
			\begin{tikzpicture}[scale=0.45]
				\begin{axis}[xlabel={Number of vertices},
					ylabel=Time(s),
					xmode=log, ymode=log, ymin=1e-1, ymax=1e2,
					log basis x={10},
					legend pos = north west,
					legend style = {legend columns = 2}]
					\addplot[dotted,mark=o, mark options={solid}]
					table [x=NV, y=PCSETUP, col sep=comma]{code/results_elasticity_s2_mg.csv};
					\addlegendentry{$\mathcal{S}_2$};
					\addplot[dashed,mark=o, mark options={solid}]
					table [x=NV, y=PCSETUP, col sep=comma]{code/results_elasticity_q2_mg.csv};
					\addlegendentry{$\mathcal{Q}_2$};
					\addplot[dotted,mark=triangle, mark options={solid}]
					table [x=NV, y=PCSETUP, col sep=comma]{code/results_elasticity_s3_mg.csv};
					\addlegendentry{$\mathcal{S}_3$};
					\addplot[dashed,mark=triangle, mark options={solid}]
					table [x=NV, y=PCSETUP, col sep=comma]{code/results_elasticity_q3_mg.csv};
					\addlegendentry{$\mathcal{Q}_3$};
					\addplot[dotted,mark=square, mark options={solid}]
					table [x=NV, y=PCSETUP, col sep=comma]{code/results_elasticity_s4_mg.csv};
					\addlegendentry{$\mathcal{S}_4$};
					\addplot[dashed,mark=square, mark options={solid}]
					table [x=NV, y=PCSETUP, col sep=comma]{code/results_elasticity_q4_mg.csv};
					\addlegendentry{$\mathcal{Q}_4$};
				\end{axis}
			\end{tikzpicture}
			\caption{MG Setup time}
			\label{elasticitymgsetup}    
		\end{subfigure}
		\begin{subfigure}[c]{0.3\textwidth}
			\begin{tikzpicture}[scale=0.45]
				\begin{axis}[xlabel={Number of vertices},
					ylabel=Time(s),
					xmode=log, ymode=log, ymin=1e-1, ymax=1e2,
					log basis x={10},
					legend pos = north west,
					legend style = {legend columns = 2}]
					\addplot[dotted,mark=o, mark options={solid}]
					table [x=NV, y=KSPSOLVE, col sep=comma]{code/results_elasticity_s2_mg.csv};
					\addlegendentry{$\mathcal{S}_2$};
					\addplot[dashed,mark=o, mark options={solid}]
					table [x=NV, y=KSPSOLVE, col sep=comma]{code/results_elasticity_q2_mg.csv};
					\addlegendentry{$\mathcal{Q}_2$};
					\addplot[dotted,mark=triangle, mark options={solid}]
					table [x=NV, y=KSPSOLVE, col sep=comma]{code/results_elasticity_s3_mg.csv};
					\addlegendentry{$\mathcal{S}_3$};
					\addplot[dashed,mark=triangle, mark options={solid}]
					table [x=NV, y=KSPSOLVE, col sep=comma]{code/results_elasticity_q3_mg.csv};
					\addlegendentry{$\mathcal{Q}_3$};
					\addplot[dotted,mark=square, mark options={solid}]
					table [x=NV, y=KSPSOLVE, col sep=comma]{code/results_elasticity_s4_mg.csv};
					\addlegendentry{$\mathcal{S}_4$};
					\addplot[dashed,mark=square, mark options={solid}]
					table [x=NV, y=KSPSOLVE, col sep=comma]{code/results_elasticity_q4_mg.csv};
					\addlegendentry{$\mathcal{Q}_4$};
				\end{axis}
			\end{tikzpicture}
			\caption{CG/MG solve time}
			\label{elasticitymgsolve}
		\end{subfigure}
		\begin{subfigure}[c]{0.3\textwidth}
			\begin{tikzpicture}[scale=0.45]
				\begin{axis}[xlabel={Number of vertices},
					ylabel=Time(s),
					xmode=log, ymode=log, ymin=1e-1, ymax=1e2,
					log basis x={10},
					legend pos = north west,
					legend style = {legend columns = 2}]
					\addplot[dotted,mark=o, mark options={solid}]
					table [x=NV, y expr=\thisrow{KSPSOLVE}+\thisrow{PCSETUP}, col sep=comma]{code/results_elasticity_s2_mg.csv};
					\addlegendentry{$\mathcal{S}_2$};
					\addplot[dashed,mark=o, mark options={solid}]
					table [x=NV, y expr=\thisrow{KSPSOLVE}+\thisrow{PCSETUP}, col sep=comma]{code/results_elasticity_q2_mg.csv};
					\addlegendentry{$\mathcal{Q}_2$};
					\addplot[dotted,mark=triangle, mark options={solid}]
					table [x=NV, y expr=\thisrow{KSPSOLVE}+\thisrow{PCSETUP}, col sep=comma]{code/results_elasticity_s3_mg.csv};
					\addlegendentry{$\mathcal{S}_3$};
					\addplot[dashed,mark=triangle, mark options={solid}]
					table [x=NV, y expr=\thisrow{KSPSOLVE}+\thisrow{PCSETUP}, col sep=comma]{code/results_elasticity_q3_mg.csv};
					\addlegendentry{$\mathcal{Q}_3$};
					\addplot[dotted,mark=square, mark options={solid}]
					table [x=NV, y expr=\thisrow{KSPSOLVE}+\thisrow{PCSETUP}, col sep=comma]{code/results_elasticity_s4_mg.csv};
					\addlegendentry{$\mathcal{S}_4$};
					\addplot[dashed,mark=square, mark options={solid}]
					table [x=NV, y expr=\thisrow{KSPSOLVE}+\thisrow{PCSETUP}, col sep=comma]{code/results_elasticity_q4_mg.csv};
					\addlegendentry{$\mathcal{Q}_4$};
				\end{axis}
			\end{tikzpicture}
			\caption{Setup+solve}
			\label{elasticitymgsetupsolve}    
		\end{subfigure}
		\caption{Solver time for planar elasticity.  The first plots give the time to set up the ASM or MG preconditioner in PETSc.  The middle column gives the time to solve the system using preconditioned conjugate gradients.  The final column gives the sum of the first two.}
		\label{fig:elasticitytime}
	\end{figure}

	\section{Conclusions}\label{sec:conc}
	We have extended the additive Schwarz theory for tensor-product methods to serendipity, proving the existence of optimal bounds for the preconditioners resulting from applying additive Schwarz methods with a low order global system and patch smoothers.
        These methods lead to preconditioned methods with iteration counts that are robust with respect to both the mesh size and polynomial degree, as confirmed both in theory and numerical examples.
        However, our analysis does not currently extend to the rational serendipity elements in~\cite{arbogast2022direct} that guarantee full accuracy in unstructured geometry.
        Moreover, although we have promising algorithmic results with Firedrake, it is likely that additional code optimizations are still possible.
        Finally, although the formulation of our method for nonsymmetric problems is straightforward, our analysis heavily uses both symmetry and coercivity, so a theoretical justification will require additional work.

	%

	
	\bibliographystyle{siamplain}
	\bibliography{references}
\end{document}

%% file: ex_shared.tex
\usepackage{booktabs}
\usepackage{amssymb,url,amsmath,graphicx,mathdots,mathtools}
\usepackage{accents}
\usepackage{tikz}
\usepackage{tikz-3dplot} 
\usepackage{mathrsfs}
\usetikzlibrary{positioning, calc, chains}
\usetikzlibrary{shapes.misc}
\usetikzlibrary{shapes.symbols}
\usetikzlibrary{shapes.geometric}
\usetikzlibrary{shapes.arrows}
\usetikzlibrary{patterns}
\usetikzlibrary{fit}
\usetikzlibrary{shadows}

\usepackage{pgfplots}
\usepackage{pgfplotstable}

\pgfplotstableset{create on use/Build/.style={create
    col/expr={\thisrow{KSPSOLVE} + \thisrow{PCSETUP}}}}

\usepackage[hypcap=true]{subcaption}
\usepackage{enumitem}

\usepackage{listings}
\definecolor{DarkBlue}{rgb}{0.00,0.00,0.55}
\definecolor{DarkRed}{rgb}{0.55,0.00,0.00}
\definecolor{DarkGreen}{rgb}{0.00,0.55,0.00}
\definecolor{Bittersweet}{rgb}{1.0, 0.44, 0.37}
\definecolor{Purple}{rgb}{0.5, 0.0, 0.5}

\newcommand{\vertiii}[1]{{\left\vert\kern-0.25ex\left\vert\kern-0.25ex\left\vert #1 
		\right\vert\kern-0.25ex\right\vert\kern-0.25ex\right\vert}}

\usepackage{mathtools}

\providecommand{\myfloor}[1]{\left \lfloor #1 \right \rfloor }

\newcommand{\citep}[1]{\cite{#1}}

\usepackage{algorithmic}
\usepackage{algorithm}

\usepackage{lipsum}
\usepackage{amsfonts}
\usepackage{graphicx}
\usepackage{epstopdf}
\usepackage{algorithmic}
\ifpdf
  \DeclareGraphicsExtensions{.eps,.pdf,.png,.jpg}
\else
  \DeclareGraphicsExtensions{.eps}
\fi

\newsiamremark{remark}{Remark}
\newsiamremark{hypothesis}{Hypothesis}
\crefname{hypothesis}{Hypothesis}{Hypotheses}
\newsiamthm{claim}{Claim}

\headers{Additive Schwarz methods for serendipity elements}{J. Marchena-Menendez, R. C. Kirby}

\title{Additive Schwarz methods for serendipity elements\thanks{Submitted to the editors 4/26/2021.
\funding{This work was supported by National Science Foundation grant 1912653.}}}

\author{Jorge Marchena-Menendez\thanks{Department of Mathematics, Baylor University, Waco, TX 
  (\email{Jorge\_Marchena1@baylor.edu}).}
\and Robert c. Kirby\thanks{Department of Mathematics, Baylor University, Waco, TX 
  (\email{Robert\_Kirby@baylor.edu}).}
}

\usepackage{amsopn}


%% file: square_scatter.tex
\begin{tikzpicture}[scale=1.5]
\draw (3.091920e-03, -1.149890e+00) -- (3.552714e-15, -1.500000e+00) -- (-3.750000e-01, -1.500000e+00) -- (-3.290000e-01, -1.142098e+00) -- cycle;
\draw (-3.290000e-01, -1.142098e+00) -- (-6.625777e-01, -1.127799e+00) -- (-7.500000e-01, -1.500000e+00) -- (-3.750000e-01, -1.500000e+00) -- cycle;
\draw (2.685014e-03, -8.173182e-01) -- (3.091920e-03, -1.149890e+00) -- (-3.290000e-01, -1.142098e+00) -- (-2.866807e-01, -8.111104e-01) -- cycle;
\draw (3.091920e-03, -1.149890e+00) -- (3.552714e-15, -1.500000e+00) -- (3.750000e-01, -1.500000e+00) -- (3.356274e-01, -1.147793e+00) -- cycle;
\draw (-2.866807e-01, -8.111104e-01) -- (-3.290000e-01, -1.142098e+00) -- (-6.625777e-01, -1.127799e+00) -- (-5.731195e-01, -8.027896e-01) -- cycle;
\draw (-6.625777e-01, -1.127799e+00) -- (-7.500000e-01, -1.500000e+00) -- (-1.125000e+00, -1.500000e+00) -- (-1.036842e+00, -1.037110e+00) -- cycle;
\draw (1.221245e-15, -5.000000e-01) -- (-2.500000e-01, -5.000000e-01) -- (-2.866807e-01, -8.111104e-01) -- (2.685014e-03, -8.173182e-01) -- cycle;
\draw (2.685014e-03, -8.173182e-01) -- (2.932087e-01, -8.131134e-01) -- (3.356274e-01, -1.147793e+00) -- (3.091920e-03, -1.149890e+00) -- cycle;
\draw (3.356274e-01, -1.147793e+00) -- (6.692832e-01, -1.132495e+00) -- (7.500000e-01, -1.500000e+00) -- (3.750000e-01, -1.500000e+00) -- cycle;
\draw (-2.500000e-01, -5.000000e-01) -- (-5.000000e-01, -5.000000e-01) -- (-5.731195e-01, -8.027896e-01) -- (-2.866807e-01, -8.111104e-01) -- cycle;
\draw (-5.731195e-01, -8.027896e-01) -- (-6.625777e-01, -1.127799e+00) -- (-1.036842e+00, -1.037110e+00) -- (-7.907383e-01, -7.910481e-01) -- cycle;
\draw (-1.036842e+00, -1.037110e+00) -- (-1.500000e+00, -1.125000e+00) -- (-1.500000e+00, -1.500000e+00) -- (-1.125000e+00, -1.500000e+00) -- cycle;
\draw (1.221245e-15, -5.000000e-01) -- (2.685014e-03, -8.173182e-01) -- (2.932087e-01, -8.131134e-01) -- (2.500000e-01, -5.000000e-01) -- cycle;
\draw (2.932087e-01, -8.131134e-01) -- (5.782581e-01, -8.101845e-01) -- (6.692832e-01, -1.132495e+00) -- (3.356274e-01, -1.147793e+00) -- cycle;
\draw (6.692832e-01, -1.132495e+00) -- (7.500000e-01, -1.500000e+00) -- (1.125000e+00, -1.500000e+00) -- (1.039504e+00, -1.039002e+00) -- cycle;
\draw (-5.000000e-01, -5.000000e-01) -- (-5.731195e-01, -8.027896e-01) -- (-7.907383e-01, -7.910481e-01) -- (-8.067497e-01, -5.775610e-01) -- cycle;
\draw (-8.067497e-01, -5.775610e-01) -- (-7.907383e-01, -7.910481e-01) -- (-1.036842e+00, -1.037110e+00) -- (-1.130910e+00, -6.669904e-01) -- cycle;
\draw (-1.130910e+00, -6.669904e-01) -- (-1.500000e+00, -7.500000e-01) -- (-1.500000e+00, -1.125000e+00) -- (-1.036842e+00, -1.037110e+00) -- cycle;
\draw (2.500000e-01, -5.000000e-01) -- (5.000000e-01, -5.000000e-01) -- (5.782581e-01, -8.101845e-01) -- (2.932087e-01, -8.131134e-01) -- cycle;
\draw (5.782581e-01, -8.101845e-01) -- (7.947752e-01, -7.938462e-01) -- (1.039504e+00, -1.039002e+00) -- (6.692832e-01, -1.132495e+00) -- cycle;
\draw (1.039504e+00, -1.039002e+00) -- (1.125000e+00, -1.500000e+00) -- (1.500000e+00, -1.500000e+00) -- (1.500000e+00, -1.125000e+00) -- cycle;
\draw (-5.000000e-01, -2.500000e-01) -- (-8.163142e-01, -2.923526e-01) -- (-8.067497e-01, -5.775610e-01) -- (-5.000000e-01, -5.000000e-01) -- cycle;
\draw (-8.163142e-01, -2.923526e-01) -- (-8.067497e-01, -5.775610e-01) -- (-1.130910e+00, -6.669904e-01) -- (-1.147730e+00, -3.369612e-01) -- cycle;
\draw (-1.147730e+00, -3.369612e-01) -- (-1.500000e+00, -3.750000e-01) -- (-1.500000e+00, -7.500000e-01) -- (-1.130910e+00, -6.669904e-01) -- cycle;
\draw (5.000000e-01, -5.000000e-01) -- (8.100809e-01, -5.752063e-01) -- (7.947752e-01, -7.938462e-01) -- (5.782581e-01, -8.101845e-01) -- cycle;
\draw (8.100809e-01, -5.752063e-01) -- (1.131244e+00, -6.658333e-01) -- (1.039504e+00, -1.039002e+00) -- (7.947752e-01, -7.938462e-01) -- cycle;
\draw (1.131244e+00, -6.658333e-01) -- (1.500000e+00, -7.500000e-01) -- (1.500000e+00, -1.125000e+00) -- (1.039504e+00, -1.039002e+00) -- cycle;
\draw (-5.000000e-01, -1.221245e-15) -- (-5.000000e-01, -2.500000e-01) -- (-8.163142e-01, -2.923526e-01) -- (-8.145140e-01, -3.981583e-03) -- cycle;
\draw (-8.145140e-01, -3.981583e-03) -- (-8.163142e-01, -2.923526e-01) -- (-1.147730e+00, -3.369612e-01) -- (-1.152156e+00, -2.818273e-03) -- cycle;
\draw (-1.152156e+00, -2.818273e-03) -- (-1.500000e+00, -3.552714e-15) -- (-1.500000e+00, -3.750000e-01) -- (-1.147730e+00, -3.369612e-01) -- cycle;
\draw (5.000000e-01, -2.500000e-01) -- (8.148399e-01, -2.865772e-01) -- (8.100809e-01, -5.752063e-01) -- (5.000000e-01, -5.000000e-01) -- cycle;
\draw (8.148399e-01, -2.865772e-01) -- (1.146942e+00, -3.293578e-01) -- (1.131244e+00, -6.658333e-01) -- (8.100809e-01, -5.752063e-01) -- cycle;
\draw (1.146942e+00, -3.293578e-01) -- (1.500000e+00, -3.750000e-01) -- (1.500000e+00, -7.500000e-01) -- (1.131244e+00, -6.658333e-01) -- cycle;
\draw (-5.000000e-01, -1.221245e-15) -- (-8.145140e-01, -3.981583e-03) -- (-8.160575e-01, 2.877888e-01) -- (-5.000000e-01, 2.500000e-01) -- cycle;
\draw (-8.145140e-01, -3.981583e-03) -- (-1.152156e+00, -2.818273e-03) -- (-1.148256e+00, 3.306678e-01) -- (-8.160575e-01, 2.877888e-01) -- cycle;
\draw (-1.152156e+00, -2.818273e-03) -- (-1.500000e+00, -3.552714e-15) -- (-1.500000e+00, 3.750000e-01) -- (-1.148256e+00, 3.306678e-01) -- cycle;
\draw (5.000000e-01, 1.221245e-15) -- (8.142700e-01, 4.218418e-03) -- (8.148399e-01, -2.865772e-01) -- (5.000000e-01, -2.500000e-01) -- cycle;
\draw (8.142700e-01, 4.218418e-03) -- (1.151819e+00, 3.143599e-03) -- (1.146942e+00, -3.293578e-01) -- (8.148399e-01, -2.865772e-01) -- cycle;
\draw (1.151819e+00, 3.143599e-03) -- (1.500000e+00, 3.552714e-15) -- (1.500000e+00, -3.750000e-01) -- (1.146942e+00, -3.293578e-01) -- cycle;
\draw (-5.000000e-01, 2.500000e-01) -- (-8.160575e-01, 2.877888e-01) -- (-8.130563e-01, 5.781219e-01) -- (-5.000000e-01, 5.000000e-01) -- cycle;
\draw (-8.160575e-01, 2.877888e-01) -- (-1.148256e+00, 3.306678e-01) -- (-1.136232e+00, 6.708038e-01) -- (-8.130563e-01, 5.781219e-01) -- cycle;
\draw (-1.148256e+00, 3.306678e-01) -- (-1.500000e+00, 3.750000e-01) -- (-1.500000e+00, 7.500000e-01) -- (-1.136232e+00, 6.708038e-01) -- cycle;
\draw (5.000000e-01, 1.221245e-15) -- (5.000000e-01, 2.500000e-01) -- (8.162106e-01, 2.923740e-01) -- (8.142700e-01, 4.218418e-03) -- cycle;
\draw (8.142700e-01, 4.218418e-03) -- (8.162106e-01, 2.923740e-01) -- (1.147660e+00, 3.369818e-01) -- (1.151819e+00, 3.143599e-03) -- cycle;
\draw (1.151819e+00, 3.143599e-03) -- (1.500000e+00, 3.552714e-15) -- (1.500000e+00, 3.750000e-01) -- (1.147660e+00, 3.369818e-01) -- cycle;
\draw (-5.000000e-01, 5.000000e-01) -- (-8.130563e-01, 5.781219e-01) -- (-7.989980e-01, 7.979026e-01) -- (-5.814866e-01, 8.129769e-01) -- cycle;
\draw (-8.130563e-01, 5.781219e-01) -- (-1.136232e+00, 6.708038e-01) -- (-1.045567e+00, 1.044970e+00) -- (-7.989980e-01, 7.979026e-01) -- cycle;
\draw (-1.136232e+00, 6.708038e-01) -- (-1.500000e+00, 7.500000e-01) -- (-1.500000e+00, 1.125000e+00) -- (-1.045567e+00, 1.044970e+00) -- cycle;
\draw (5.000000e-01, 2.500000e-01) -- (8.162106e-01, 2.923740e-01) -- (8.065306e-01, 5.773707e-01) -- (5.000000e-01, 5.000000e-01) -- cycle;
\draw (8.162106e-01, 2.923740e-01) -- (8.065306e-01, 5.773707e-01) -- (1.130798e+00, 6.669067e-01) -- (1.147660e+00, 3.369818e-01) -- cycle;
\draw (1.147660e+00, 3.369818e-01) -- (1.500000e+00, 3.750000e-01) -- (1.500000e+00, 7.500000e-01) -- (1.130798e+00, 6.669067e-01) -- cycle;
\draw (-2.500000e-01, 5.000000e-01) -- (-5.000000e-01, 5.000000e-01) -- (-5.814866e-01, 8.129769e-01) -- (-2.954056e-01, 8.139860e-01) -- cycle;
\draw (-5.814866e-01, 8.129769e-01) -- (-7.989980e-01, 7.979026e-01) -- (-1.045567e+00, 1.044970e+00) -- (-6.746200e-01, 1.137396e+00) -- cycle;
\draw (-1.045567e+00, 1.044970e+00) -- (-1.125000e+00, 1.500000e+00) -- (-1.500000e+00, 1.500000e+00) -- (-1.500000e+00, 1.125000e+00) -- cycle;
\draw (5.000000e-01, 5.000000e-01) -- (5.719579e-01, 8.019055e-01) -- (7.902823e-01, 7.906900e-01) -- (8.065306e-01, 5.773707e-01) -- cycle;
\draw (8.065306e-01, 5.773707e-01) -- (7.902823e-01, 7.906900e-01) -- (1.036568e+00, 1.036889e+00) -- (1.130798e+00, 6.669067e-01) -- cycle;
\draw (1.130798e+00, 6.669067e-01) -- (1.500000e+00, 7.500000e-01) -- (1.500000e+00, 1.125000e+00) -- (1.036568e+00, 1.036889e+00) -- cycle;
\draw (-1.221245e-15, 5.000000e-01) -- (-2.500000e-01, 5.000000e-01) -- (-2.954056e-01, 8.139860e-01) -- (-4.218073e-03, 8.170749e-01) -- cycle;
\draw (-2.954056e-01, 8.139860e-01) -- (-5.814866e-01, 8.129769e-01) -- (-6.746200e-01, 1.137396e+00) -- (-3.378453e-01, 1.148806e+00) -- cycle;
\draw (-6.746200e-01, 1.137396e+00) -- (-7.500000e-01, 1.500000e+00) -- (-1.125000e+00, 1.500000e+00) -- (-1.045567e+00, 1.044970e+00) -- cycle;
\draw (2.500000e-01, 5.000000e-01) -- (5.000000e-01, 5.000000e-01) -- (5.719579e-01, 8.019055e-01) -- (2.850630e-01, 8.101810e-01) -- cycle;
\draw (5.719579e-01, 8.019055e-01) -- (6.618148e-01, 1.127268e+00) -- (1.036568e+00, 1.036889e+00) -- (7.902823e-01, 7.906900e-01) -- cycle;
\draw (1.036568e+00, 1.036889e+00) -- (1.500000e+00, 1.125000e+00) -- (1.500000e+00, 1.500000e+00) -- (1.125000e+00, 1.500000e+00) -- cycle;
\draw (-1.221245e-15, 5.000000e-01) -- (2.500000e-01, 5.000000e-01) -- (2.850630e-01, 8.101810e-01) -- (-4.218073e-03, 8.170749e-01) -- cycle;
\draw (-4.218073e-03, 8.170749e-01) -- (-2.954056e-01, 8.139860e-01) -- (-3.378453e-01, 1.148806e+00) -- (-4.680761e-03, 1.149614e+00) -- cycle;
\draw (-3.378453e-01, 1.148806e+00) -- (-6.746200e-01, 1.137396e+00) -- (-7.500000e-01, 1.500000e+00) -- (-3.750000e-01, 1.500000e+00) -- cycle;
\draw (2.850630e-01, 8.101810e-01) -- (3.273672e-01, 1.141133e+00) -- (6.618148e-01, 1.127268e+00) -- (5.719579e-01, 8.019055e-01) -- cycle;
\draw (6.618148e-01, 1.127268e+00) -- (7.500000e-01, 1.500000e+00) -- (1.125000e+00, 1.500000e+00) -- (1.036568e+00, 1.036889e+00) -- cycle;
\draw (-4.218073e-03, 8.170749e-01) -- (-4.680761e-03, 1.149614e+00) -- (3.273672e-01, 1.141133e+00) -- (2.850630e-01, 8.101810e-01) -- cycle;
\draw (-4.680761e-03, 1.149614e+00) -- (-3.552714e-15, 1.500000e+00) -- (-3.750000e-01, 1.500000e+00) -- (-3.378453e-01, 1.148806e+00) -- cycle;
\draw (3.273672e-01, 1.141133e+00) -- (6.618148e-01, 1.127268e+00) -- (7.500000e-01, 1.500000e+00) -- (3.750000e-01, 1.500000e+00) -- cycle;
\draw (-4.680761e-03, 1.149614e+00) -- (-3.552714e-15, 1.500000e+00) -- (3.750000e-01, 1.500000e+00) -- (3.273672e-01, 1.141133e+00) -- cycle;
\end{tikzpicture}